\documentclass[a4paper,reqno]{amsart}
\usepackage[utf8]{inputenc}
\usepackage{amsmath,amssymb,amsthm,amsfonts}
\usepackage{enumitem}
\usepackage{esint}

\newtheorem{theorem}{Theorem}[section]
\newtheorem{corollary}[theorem]{Corollary}
\newtheorem{lemma}[theorem]{Lemma}
\newtheorem{proposition}[theorem]{Proposition}
\theoremstyle{definition}
\newtheorem{definition}[theorem]{Definition}
\theoremstyle{remark}
\newtheorem{remark}[theorem]{Remark}
\newtheorem{example}[theorem]{Example}

\DeclareMathOperator{\dist}{dist}
\def\reals{\mathbf{R}}

\def\complexes{\mathbf{C}}
\def\naturalnumbers{\mathbf{N}}
\def\integers{\mathbf{Z}}

\def\E{\mathbb{E}\,} \def\P{\mathbb{P}}
\def\T{\mathcal{T}}
\def\M{\mathcal{M}}
\def\H{\mathbf{H}}
\def\mhyp{m_{hyp}}
\def\W{W}
\def\Wper{W_{per}}
\newcommand{\mutorus}[1]{\mu_{#1,S^1}}

\author[J.  Junnila]{Janne Junnila}
\thanks{The first author was supported by the Doctoral Programme in Mathematics and Statistics at University of Helsinki}
\address{University of Helsinki, Department of Mathematics and Statistics, P.O. Box 68, FIN-00014 University of Helsinki, Finland}
\email{janne.junnila@helsinki.fi}

\author[E. Saksman]{Eero Saksman}
\thanks{The second author was supported by The Finnish Centre of Excellence (CoE) in Analysis and Dynamics Research}
\address{University of Helsinki, Department of Mathematics and Statistics, P.O. Box 68, FIN-00014 University of Helsinki, Finland}
\email{eero.saksman@helsinki.fi}


\title{The uniqueness of the Gaussian multiplicative chaos revisited}

\begin{document}

\begin{abstract}
We consider Gaussian multiplicative chaos measures defined in a general setting of metric measure spaces. Uniqueness results are obtained, verifying that different sequences of approximating Gaussian fields lead to the same chaos measure. Specialized to Euclidean spaces, our setup covers both the subcritical chaos and the critical chaos, actually extending to all non-atomic Gaussian chaos measures.
\end{abstract}

\maketitle

\section{Introduction}\label{sec:introduction}

The theory of multiplicative chaos was created by Kahane \cite{kahane1985,kahane1987positive} in the 1980's in order to obtain a continuous counterpart of the multiplicative cascades, which were proposed by Mandelbrot in early 1970's as a model for turbulence. During the last 10 years there has been a new wave of interest on multiplicative chaos, due to e.g.\ its important connections to Stochastic Loewner Evolution \cite{astala2011random,sheffield2010conformal,duplantier2011schramm}, quantum gravity \cite{duplantier2009duality,duplantier2011liouville,berestycki2014liouville,miller2013quantum}, models in finance and turbulence \cite[Section 5]{rhodes2014gaussian}, and the hypothetical statistical behaviour of the Riemann zeta function over the critical line \cite{fyodorov2014freezing}.

In Kahane's original theory one considers a sequence of a.s.\ continuous and centered Gaussian fields $X_n$ that can be thought of as approximations of a (possibly distribution valued) Gaussian field $X$. The fields are defined on some metric measure space $(\T,\lambda)$ and the increments $X_{n+1}-X_n$ are assumed to be independent. One may then define the random measures $\mu_n$  on $\T$ by setting
$$
\mu_n(dx):=\exp(X_n(x)-\frac{1}{2}\E X_n(x)^2)\lambda(dx).
$$
In this situation basic martingale theory verifies that almost surely there exists a (random) limit measure $\mu=\lim_{n\to\infty}\mu_n$, where the convergence is understood in the ${\rm weak^*}$-sense. The measure $\mu$ is called the \emph{multiplicative chaos} defined by $X$ (or rather by the sequence ($X_n$)), and Kahane shows that under suitable conditions the limit does not depend on the choice of the approximating sequence ($X_n$). However, the limit may well reduce to the zero measure almost surely.

We next recall some of the most important cases of multiplicative chaos in the basic setting where $\T$ is a subset of a Euclidean space, say $\T=[0,1]^d$, and $\lambda$ is the Lebesgue measure. Especially we assume that the limit field $X$ is log-correlated, i.e.\ it has the covariance
\begin{equation}\label{eq:log_covariance}
C_X(x,y)=2d\beta^2\log |x-y|+ G(x,y),\qquad x,y\in \T,
\end{equation}
where $G$ is a continuous and bounded function.
This means that $X$ is essentially a multiple of (the trace of) a Gaussian free field (GFF).

Assuming that the $X_n$ are nice approximations of the field $X$ as explained above, Kahane's theory yields that in case $\beta\in (0,1)$ the convergence $\mu_n \overset{w^*}{\to} \mu_\beta$ takes place almost surely and the obtained chaos $\mu_\beta$ is non-trivial.
It is an example of \emph{subcritical Gaussian chaos}, and, as we shall soon recall in more detail, in this normalisation $\beta=1$ appears as a critical value.

In order to give a more concrete view of the chaos we take a closer look at a particularly important example of approximating Gaussian fields in the case where $d=1$ and $\mu$ is the so-called exactly scale invariant chaos due to Bacry and Muzy \cite{bacry2003log}, \cite[p. 15]{rhodes2014gaussian}. Consider the hyperbolic white noise $W$ in the upper half plane $\reals^2_+$ so that
$\E \W(A_1)\W(A_2)=m_{\text{hyp}}(A_1 \cap A_2)$ for Borel subsets $A_1,A_2\in \reals^2_+$ with compact closure in $\reals^2_+$. Above $dm_{\text{hyp}}=y^{-2}dx \, dy$ denotes the hyperbolic measure in the upper half plane. For every $t>0$ consider the set
\begin{equation}\label{eq:At}
A_t(x) := \{(x',y')\in \reals^2_+ : y' \ge \max(e^{-t},2|x' - x|) \text{\ and } |x' - x|\leq 1/2\}
\end{equation}
and define the field $X_t$ on $[0,1]$ by setting
\[X_t(x) := \sqrt{2d}\W(A_{t}(x)).\]
Note that the sets $A_t(x)$ are horizontal translations of the set $A_t(0)$.
One then defines the subcritical exactly scale invariant chaos by setting
\begin{equation}\label{eq:scr}
  d\mu_\beta(x)\overset{\textrm{a.s}}{:=}\lim_{t\to\infty}\exp\big( \beta X_{t}(x)-\frac{\beta^2}{2}\E (X_{t}(x))^2\big) \, dx \quad \text{for } \beta < 1.
\end{equation}

If $\beta = 1$, the above limit equals the zero measure almost surely. To construct the exactly scaling chaos measure at criticality $\beta=1$, one has to perform a non-trivial normalization as follows:
\begin{equation}\label{eq:cr}
  d\mu_{1}(x):=\lim_{t\to\infty}\sqrt{t}\exp\big( X_{t}(x)-\frac{1}{2}\E (X_{t}(x))^2\big) \, dx,
\end{equation}
where the limit now exists in probability.

The need of a nontrivial normalisation at the critical parameter value in \eqref{eq:cr} has been observed in many analogous situations before, e.g. \cite{bramson1983convergence,yor2001some}. A convergence result analogous to \eqref{eq:cr} was proven by Aidekon and Shi in the important work \cite{aidekon2014seneta} in the case of Mandelbrot chaos measures that can be thought of as a discrete analogue of continuous chaos. Independently C. Webb \cite{webb2011exact} obtained the corresponding result (with convergence in distribution) for the Gaussian cascades (\cite{aidekon2014seneta} and \cite{webb2011exact} considered the total mass, but the convergence of the measures can then be verified without too much work). Finally, Duplantier, Rhodes, Sheffield and Vargas \cite{duplantier2014critical,duplantier2014renormalization} established \eqref{eq:cr} for a class of continuous Gaussian chaos measures including the exactly scaling one. We refer to \cite{rhodes2014gaussian,duplantier2014log} for a much more thorough discussion of chaos measures and their applications, as well as for further references on the topic.

An important issue is to understand when the obtained chaos measure is independent of the choice of the approximating fields $X_n$. As mentioned before, Kahane's seminal work contained some results in this direction. Robert and Vargas \cite{robert2010gaussian} addressed the uniqueness question in the case of subcritical log-correlated fields \eqref{eq:log_covariance} for convolution approximations $X_n = \phi_{\varepsilon_n} * X$. 
Duplantier's and Sheffield's paper \cite{duplantier2011liouville} gives uniqueness results for particular approximations  of the 2-dimensional GFF. More general results developing the method of \cite{robert2010gaussian} are contained in the reviews \cite{rhodes2014gaussian} due to Rhodes and Vargas, and in \cite{david2015liouville} the method is also applied in a special case of critical chaos. Their conditions are very similar to ours in this paper.  Another approach is contained in the paper of Shamov \cite[Sections 7, 8]{shamov2014gaussian}. The techniques of the latter paper are based on an interesting new characterisation of chaos measures, which is applicable in the subcritical range. Finally, Berestycki \cite{Be2015} has a new simple proof for convolution approximations, again in the subcritical regime.

In the present paper we develop an alternative approach to this important question of independence on the approximation fields used. Our simple idea uses a specifically tailored auxiliary field added to the original field in order to obtain comparability directly from Kahane's convexity inequality, and the choice is made so that in the limit the effect of the auxiliary field vanishes. The approach is outlined before the actual proof in the beginning of Section~\ref{sec:auxiliary}. One obtains a unified result that applies in general to chaos measures obtained via an arbitrary normalization, the only requirement is that the chaos measure is non-atomic almost surely. Especially, the case of the classical critical chaos is covered, and moreover our results apply also to a class of chaos measures that lie between the critical and supercritical ones, which one expects to be useful in the study of finer properties of the critical chaos itself.

Our basic result considers the following situation: Let $(X_n)$ and $(\widetilde{X}_n)$ be two sequences of H\"older-regular Gaussian fields (see Section~\ref{sec:notation} for the precise definition) on a compact doubling metric space $(\T,d)$. Assume that for each $n \ge 1$ we have a non-negative Radon reference measure $\rho_n$ defined on $\T$. Define the measures
\[d\mu_n(x) := e^{X_n(x) - \frac{1}{2} \E[X_n(x)^2]} \, d\rho_n(x)\]
for all $n \ge 1$. The measures $\widetilde{\mu}_n$ are defined analogously by using the fields $\widetilde{X}_n$ instead.

\begin{theorem}\label{thm:uniqueness}
  Let $C_n(x,y)$ and $\widetilde{C}_n(x,y)$ be the covariance functions of the fields $X_n$ and $\widetilde{X}_n$ respectively. Assume that the random measures $\widetilde{\mu}_n$ converge in distribution to an almost surely non-atomic random measure $\widetilde{\mu}$ on $\T$. Moreover, assume that the covariances $C_n$ and $\widetilde{C}_n$ satisfy the following two conditions:
  There exists a constant $K > 0$ such that
  \begin{equation}\label{ehto1}
    \sup_{x,y \in \T} |C_n(x,y) - \widetilde{C}_n(x,y)| \le K < \infty \quad \text{for all } n \ge 1,
  \end{equation}
  and
  \begin{equation}\label{ehto2}
    \lim_{n \to \infty} \sup_{d(x,y) > \delta} |C_n(x,y) - \widetilde{C}_n(x,y)| = 0 \quad \text{for every } \delta > 0.
  \end{equation}
Then the measures $\mu_n$ converge in distribution to the same random measure $\widetilde{\mu}$.
\end{theorem}

\begin{remark}\label{rmk:uniqueness_non_compact}
  For simplicity we have stated the above theorem and will give the proof in the setting of a compact space $\T$.
  Similar results are obtained for non-compact $\T$ by standard localization.
  For example assume that $\T$ has an exhaustion $\T = \bigcup_{n=1}^\infty K_n$ with compacts $K_1 \subset K_2 \subset \dots \subset \T$, such that every compact $K \subset \T$ is eventually contained in some $K_n$.
  Then if the assumptions of Theorem~\ref{thm:uniqueness} are valid for the restrictions to each $K_n$, the claim also holds for $\T$, where now weak convergence is defined using compactly supported test functions.
\end{remark}

\noindent The proof of the above theorem is contained in Section~\ref{sec:proof}, where it is also noted that one may somewhat loosen the condition \eqref{ehto1}, see Remark~\ref{rmk:loosen}. We refer to Section~\ref{sec:notation} for precise definitions of convergence in the space of measures and other needed prerequisities.

Section~\ref{sec:inprobability} addresses the question when the convergence in Theorem~\ref{thm:uniqueness} can be lifted to convergence in probability (or in $L^p$). The basic underlying assumption is that this is the case for some other approximation sequence that has a martingale structure -- a condition which is often met in applications.

In Section~\ref{sec:auxiliary} we state some consequences for convolution approximations (see Corollaries~\ref{cor:convolutions} and~\ref{cor:convolutions_circle}). Moreover, convergence results are stated for mildly perturbed fields.

Finally, Section~\ref{sec:application} illustrates the use of the results of the previous sections. This is done via taking a closer look at the fundamental critical chaos on the unit circle, obtained from the GFF defined via the Fourier series
\[X(x)=2\sqrt{\log 2} A_0 + \sqrt{2}\sum_{k=1}^\infty k^{-1/2}\big(A_k\sin (2\pi k x)+B_k\cos (2\pi k x)\big)\quad \textrm{for}\quad x\in [0,1),\]
where the $A_n, B_n$ are independent standard Gaussians. In \cite{astala2011random} the corresponding subcritical Gaussian chaos was constructed using martingale approximates defined via the periodic hyperbolic white noise. We shall consider four different approximations of $X$:

\begin{enumerate}
  \item[1.] $X_{1,n}$ is the approximation of $X$ obtained by cutting the periodic hyperbolic white noise construction of $X$ on the level $1/n$.

  \item[2.] $X_{2,n}(x) = 2\sqrt{\log 2}A_0 + \sqrt{2}\sum_{k=1}^n k^{-1/2}\big(A_k\sin (2\pi k x)+B_k\cos(2\pi k x)\big)$ for $x \in [0,1)$.

  \item[3.] $X_{3,n}=\phi_{1/n}*X,$ where $\phi$ is a mollifier function defined on $\T$ that satisfies some weak conditions.

  \item[4.] $X_{4,n}$ is obtained as the $n$th partial sum of a vaguelet decomposition of $X$.
\end{enumerate}

\begin{theorem}\label{co:1} For all $j=1,\ldots,3$ the random measures
  \[\sqrt{\log n}\exp\big( X_{j,n}(x)-\frac{1}{2}\E (X_{j,n}(x))^2\big) \, dx\]
  converge as $n\to\infty$ in probability to the same nontrivial random measure $\mutorus{1}$ on $\T$, which is the fundamental critical measure on $\T$. The convergence actually takes place in $L^p(\Omega)$ for every $0 < p < 1$. The same holds for the vaguelet decomposition $X_{4,n}$ with the normalization $\sqrt{n \log 2}$ instead of $\sqrt{\log n}$.
\end{theorem}

We refer to Section~\ref{sec:application} for the precise definitions of the approximations used above.
Theorem~\ref{co:1} naturally holds true in the subcritical case if above $X_{j,n}$ is replaced by $\beta X_{j,n}$ with $\beta\in (0,1)$, and one removes the factor $\sqrt{\log n}$. We denote the limit measure by $\mutorus{\beta}$.

\medskip

\noindent\emph{Acknowledgements.}\quad We thank Dario Gasbarra for useful discussions in connection with Lemma~\ref{lemma:measure_equal_in_prob} and Christian Webb for many valuable remarks on the manuscript.

\section{Notation and basic definitions}\label{sec:notation}

\noindent A metric space is \emph{doubling} if there exists a constant $M > 0$ such that any ball of radius $\varepsilon > 0$ can be covered with at most $M$ balls of radius $\varepsilon/2$. In this work we shall always consider a doubling compact metric space $(\T,d)$. We denote by $\M^+$ the space of (positive) Radon measures on $\T$.
The space $\M$ of real-valued Radon measures on $\T$ can be given the weak$^*$-topology by interpreting it as the dual of $C(\T)$. We then give $\M^+ \subset \M$ the subspace topology.

The space $\M^+$ is metrizable (which is not usually the case for the full space $\M$), for example by using the Kantorovich--Rubinstein metric defined by
\[d(m, m') := \sup \left\{\int_{\T} f(x) \, d(m - m')(x) : f \colon \T \to \reals \text{ is 1-Lipschitz}\right\}.\]
For a proof see \cite[Theorem 8.3.2]{bogachev2007measure}.

Let $\mathcal{P}(\M^+)$ denote the space of Radon probability measures on $\M^+$. One should note that Borel probability measures and Radon probability measures coincide in this situation, as well as in the case of $\mathcal{P}(\T)$, since we are dealing with Polish spaces. Let $(\Omega, \mathcal{F}, \P)$ be a fixed probability space. 
We call a measurable map $\mu \colon \Omega \to \M^+$ a {\it random measure}\ on $\T$.
For a given random measure $\mu$ the push-forward measure $\mu_* \P \in \mathcal{P}(\M^+)$ is called the distribution of $\mu$ and we say that a family of random measures $\mu_n$ converges in distribution if the measures $\mu_{n*} \P$ converge weakly in $\mathcal{P}(\M^+)$ (i.e.\ when evaluated against bounded continuous functions $\mathcal{P}(\M^+) \to \reals$). In order to check the convergence in distribution, it is enough to verify that
\[\mu_n(f) := \int f(x) \, d\mu_n(x)\]
converges in distribution for every $f \in C(\T)$.

A stronger form of convergence is the following: We say that a sequence of random measures $(\mu_n)$ converges {\it weakly in $L^p$}\ to a random measure $\mu$ if for all $f \in C(\T)$ the random variable $\int f(x) \, d\mu_n(x)$ converges in $L^p(\Omega)$ to $\int f(x) \, d\mu(x)$. This obviously implies the convergence $\mu_n \to \mu$ in distribution.

A (pointwise defined) Gaussian field $X$ on $\T$ is a random process indexed by $\T$ such that $(X(t_1), \dots, X(t_n))$ is a multivariate Gaussian random variable for every $t_1,\dots,t_n \in \T$, $n \ge 1$.
We will assume that all of our Gaussian fields are centered unless otherwise stated.

\begin{definition}\label{def:holder_regular}
  A (centered) Gaussian field $X$ on a compact metric space $\T$ is \textit{H\"older-regular} if the map $(x,y) \mapsto \sqrt{\E|X(x) - X(y)|^2}$ is $\alpha$-H\"older continuous on $\T \times \T$ for some $\alpha > 0$.
\end{definition}

\begin{lemma}\label{lemma:holder_regular_cont}
  The realizations of any H\"older-regular Gaussian field on $\T$ can be chosen to be almost surely $\beta$-H\"older continuous with some $\beta > 0$.
\end{lemma}
\begin{proof}
  This is an immediate consequence of Dudley's theorem (See for instance \cite[Theorem 1.3.5]{adler2009random}.) and the fact that our space is doubling.
\end{proof}

\begin{remark}
  By Dudley's theorem the conclusion of Lemma~\ref{lemma:holder_regular_cont} would be valid under much less restrictive assumptions on the covariance, and most of the results of the present paper could be reformulated accordingly.
\end{remark}

\noindent Assume that we are given a sequence of H\"older-regular Gaussian fields $(X_n)$ on $\T$ and also a sequence of measures $\rho_n \in \M^+$.
Define for all $n \ge 1$ a random measure $\mu_n \colon \Omega \to \M^+$ by setting
\begin{equation}\label{eq:yhteys}
\mu_n(f) := \int_{\T} f(x) e^{X_{n}(x) - \frac{1}{2} \E[X_n{(x)}^2]} \, d\rho_n(x),
\end{equation}
for all $f \in C(\T)$.
In the case where the measures $\mu_n$ converge in distribution to a random measure $\mu \colon \Omega \to \M^+$, we call $\mu$ a \emph{Gaussian multiplicative chaos} (GMC) associated with the families $X_n$ and $\rho_n$. We call the sequence of measures $\rho_n$ a \emph{normalizing sequence}. In the standard models of subcritical and critical chaos the typical choices are $\rho_n := \lambda$ and $\rho_n := C \sqrt{n} \lambda$ (or $\rho_n := C \sqrt{\log n} \lambda$), respectively, where $\lambda$ stands for the Lebesgue measure.

Unless otherwise stated, when comparing the limits of two sequences of random measures $(\mu_n)$ and $(\widetilde{\mu}_n)$, we will always use the same normalizing sequence $(\rho_n)$ to construct both $\mu_n$ and $\widetilde{\mu}_n$.

Lastly we recall the following fundamental convexity inequality due to Kahane \cite{kahane1985}.
\begin{lemma}
  Assume that $X$ and $Y$ are two H\"older-regular fields such that the covariances satisfy $C_X(s,t) \ge C_Y(s,t)$ for all $s,t \in \T$. Then for every concave function $f \colon [0,\infty) \to [0,\infty)$ we have
      \[\E\Bigg[f\Big(\int_{\T} e^{X(t) - \frac{1}{2}\E[X(t)^2]} \, d\rho(t)\Big)\Bigg] \le \E\Bigg[f\Big(\int_{\T} e^{Y(t) - \frac{1}{2} \E[Y(t)^2]} \, d\rho(t)\Big)\Bigg]\]
  for all $\rho \in \M^+$.
\end{lemma}

\section{Convergence and uniqueness: Proof of Theorem~\ref{thm:uniqueness}}\label{sec:proof}

\noindent In this section we prove Theorem~\ref{thm:uniqueness}.
The simple idea of the proof is as follows:
We construct a sequence of auxiliary fields $Y_\varepsilon$ (see especially Lemma~\ref{lemma:construction_of_Z}) that we add on top of the fields $X_n$ in order to ensure that the covariance of $X_n + Y_\varepsilon$ dominates the covariance of $\widetilde{X}_n$ pointwise. The fields $Y_\varepsilon$ become fully decorrelated as $\varepsilon \to 0$, and their construction relies on the non-atomicity of the random measure $\widetilde{\mu}$.
After these preparations one may finish by a rather standard application of Kahane's chaos comparison inequality.

The next two lemmata are almost folklore, but we provide proofs for completeness.

\begin{lemma}\label{lemma:auxiliary_function}
  Let $(\mu_n)$ be a tight sequence of random measures. Then there exists a function $h \colon [0,\infty) \to [0,\infty)$ that has the following properties:
  \begin{enumerate}
    \item functions $h, h^2$ and $h^4$ are increasing and concave with $h(0) = 0$ and $\lim_{x \to \infty} h(x) = \infty$,
    \item $h$ satisfies $\min(1,x) h(y) \le h(xy) \le \max(1,x) h(y)$, and
    \item $\sup_{n \ge 1} \E h(\mu_n(\T))^4 < \infty$.
  \end{enumerate}
\end{lemma}

\begin{proof}
  First of all, by the definition of tightness one may easily pick an increasing $g \colon [0,\infty) \to [1,\infty)$ with $\lim_{x \to \infty} g(x) = \infty$ such that $\sup_{n \ge 1} \E[g(\mu_n(\T))] < \infty$.
  Namely, let $0 = t_0 \le t_1 \le t_2 \le \dots$ be an increasing sequence of real numbers such that
  $\sup_{n \ge 1} \P[\mu_n(\T) \ge t_k] \le k^{-2}$
  for all $k \ge 1$ and set $g(x) = \sum_{k=0}^\infty \chi_{[t_k,\infty)}$.
  One may choose a concave function $\widetilde{h}$ that is majorized by $g$ and satisfies both $\widetilde{h}(0) = 0$ and $\lim_{x \to \infty} h(x) = \infty$. Finally, set $h(x) := (\widetilde{h}(x))^{1/4}$. Condition (3) follows, and (2) is then automatically satisfied by concavity. Since compositions of non-negative concave functions remain concave we obtain (1) as well.
\end{proof}

\begin{lemma}\label{lemma:convergence}
  For $n\geq 1$ let $X_n$ and $\widetilde{X}_n$ be H\"older-regular Gaussian fields on $\T$ with covariance functions
  $C_n(x,y)$ and $\widetilde{C}_n(x,y)$. Define the random measures $\mu_n$ and $\widetilde{\mu}_n$ using the fields $X_n$ and $\widetilde{X}_n$, respectively.
  Assume that there exists a constant $K > 0$ such that
  \[\sup_{x,y \in \T} (\widetilde{C}_n(x,y) - C_n(x,y)) \le K < \infty\]
  for all $n\ge1$ and that the family $(\widetilde{\mu}_n)$ is tight {\rm (}in $\mathcal{P}(\M^+)${\rm )}.
  Then also the family $(\mu_n)$ is tight.
\end{lemma}

\begin{proof}
   By the Banach--Alaoglu theorem
   it is enough to check that
   \[\lim_{u \to \infty} \sup_{n\geq 1}\P[\mu_n(\T) > u] = 0.\]
   Since $\lim_{u\to\infty}h(u)=\infty$, it suffices to verify that $\sup_{n\ge 1} \E h(\mu(\T))<\infty, $ where $h$ is the concave function given by Lemma~\ref{lemma:auxiliary_function} for the tight sequence $\widetilde\mu_n$. Pick an independent standard Gaussian $G$. By our assumption the covariance of the field $X'_n:=X_n + K^{1/2}G$ dominates that of the field $\widetilde{X}_n$, and if the random measure $\mu'_n$ is defined by using the field $X'_n,$ we obtain by Kahane's concavity inequality
\[
  \E(h(\mu'_n(\T)))^2 \le \E(h(\widetilde{\mu}_n(\T)))^2 \le c \quad \text{for any}\; n \ge 1
  \]
for some constant $c > 0$ not depending on $n$.

Since $\mu'_n=e^{K^{1/2}G-K/2}\mu_n$ the properties (2) and (3) of Lemma~\ref{lemma:auxiliary_function} enable us to estimate for all $n \ge 1$ that
\begin{align}
 \E h(\mu_n(\T))  &= \E\, h(e^{-K^{1/2}G+K/2}\mu'_n(\T))\le \E\big( \max(1,e^{-K^{1/2}G+K/2}) h(\mu'_n(\T))\big)\nonumber \\
                  &\le  \big( \E (\max(1,e^{-K^{1/2}G+K/2}))^2\big)^{1/2} \big( \E (h(\widetilde{\mu}_n(\T)))^2\big)^{1/2} \le c' \sqrt{c},\nonumber
\end{align}
for some $c' > 0$.
\end{proof}

Our proof of Theorem~\ref{thm:uniqueness} is based on the following two lemmas.
\begin{lemma}\label{lemma:diagonal_approximation}
  Let $(X_n) $ and $(\widetilde{X}_n)$ be two sequences of H\"older-regular Gaussian fields on $\T$. Assume that there exists a constant $K > 0$ such that the covariances satisfy
  \[\sup_{x,y \in \T} |\widetilde{C}_n(x,y) - C_n(x,y)| \le K < \infty
  \]
  for all $n\geq 1$. Assume also that both of the corresponding sequences of random measures $(\mu_n)$ and $(\widetilde{\mu}_n)$ converge in distribution to measures $\mu$ and $\widetilde{\mu}$ respectively, and that $\widetilde{\mu}$ is almost surely non-atomic. Then also $\mu$ is almost surely non-atomic.
\end{lemma}

\begin{proof}
  Let $G$ be an independent centered Gaussian random variable with variance $\E G^2=K$. Then the covariance of the field $X_n + G$ dominates that of the field $\widetilde{X}_n$. Define a field $U_n(x,y) := X_n(x) + X_n(y) + 2G$ on the product space $\T \times \T$. Its covariance is given by
  \begin{multline*}
    \E[U_n(x,y)U_n(x',y')] = \E[X_n(x)X_n(x')] + \E[X_n(y)X_n(y')] + \E[X_n(x)X_n(y')] \\
    + \E[X_n(y)X_n(x')] + 4K,
  \end{multline*}
  and therefore dominates the covariance of the field $V_n(x,y) := \widetilde{X}_n(x) + \widetilde{X}_n(y)$ given by
  \begin{multline*}
    \E[V_n(x,y)V_n(x',y')] = \E[\widetilde{X}_n(x)\widetilde{X}_n(x')] + \E[\widetilde{X}_n(y)\widetilde{X}_n(y')] + \E[\widetilde{X}_n(x)\widetilde{X}_n(y')] \\
    + \E[\widetilde{X}_n(y)\widetilde{X}_n(x')].
  \end{multline*}
  Define a measure $\rho_n'$ on $\T \times \T$ by setting
  \[d\rho_n'(x,y) = \sum_{k=1}^m f_k(x) f_k(y) e^{\E[X_n(x) X_n(y)]} d(\rho_n \otimes \rho_n)(x,y),\]
  where $f_1,\dots,f_m \in C(\T)$ is a fixed arbitrary finite collection of continuous functions.
  Observe that $\rho'_n$ is absolutely continuous with respect to the measure $\rho_n \otimes \rho_n$.
  By Kahane's convexity inequality we have
\begin{align*}
  & \E\left[h\big(\int_{\T \times \T} e^{U_n(x,y) - \frac{1}{2} \E[U_n(x,y)^2]} \, d\rho_n'(x,y)\big)\right] \\
  \le \; & \E\left[h\big(\int_{\T \times \T} e^{V_n(x,y) - \frac{1}{2} \E[V_n(x,y)^2]} \, d\rho'_n(x,y)\big)\right],\
\end{align*}
  where $h$ is the function from Lemma~\ref{lemma:auxiliary_function}  chosen for the sequence $(\widetilde \mu_n)$.
By Lemma~\ref{lemma:auxiliary_function} the left hand side is larger than
  \begin{align*}
    & \E\bigg[\min(1,e^{2G - 2K}) h\Big(\sum_{k=1}^m \int f_k(x) f_k(y) e^{X_n(x) + X_n(y) - \frac{1}{2} \E[X_n(x)^2] - \frac{1}{2}\E[X_n(y)^2]} \\
    &\phantom{aaaaaaaaaaaaaaaaaaaa} \cdot d(\rho_n \otimes \rho_n)(x,y)\Big)\bigg] \; \ge  \;
     A \E[h\big(\sum_{k=1}^m \mu_n(f_k)^2\big)]
  \end{align*}
  for some constant $A$ that depends only on $K$.
  Similarly the right-hand side is at most
  \begin{align*}
    & \E h\Bigg(\sum_{k=1}^m \int f_k(x) f_k(y) \exp\bigg(\widetilde{X}_n(x) + \widetilde{X}_n(y) - \frac{1}{2} \E[\widetilde{X}_n(x)^2] - \frac{1}{2}\E[\widetilde{X}_n(y)^2] \\
    & \phantom{aaaaaaaaaaa} - \E[\widetilde{X}_n(x) \widetilde{X}_n(y)] + \E[X_n(x) X_n(y)]\bigg) \, d(\rho_n \otimes \rho_n)(x,y)\Bigg)  \\
    & \le \max(1,e^{K}) \E h\Bigg(\sum_{k=1}^m \int f_k(x) f_k(y) \exp\bigg(\widetilde{X}_n(x) + \widetilde{X}_n(y) \\
    & \phantom{aaaaaaaaaaa} - \frac{1}{2} \E[\widetilde{X}_n(x)^2] - \frac{1}{2}\E[\widetilde{X}_n(y)^2]\bigg) \, d(\rho_n \otimes \rho_n)(x,y)\Bigg)  \\
    & =\max(1,e^K) \E h\Big(\sum_{k=1}^m \widetilde{\mu}_n(f_k)^2\Big).
  \end{align*}
 Thus we have the inequality
  \[\E h\Big(\sum_{k=1}^m \mu_n(f_k)^2\Big) \le c \E h\Big(\sum_{k=1}^m \widetilde{\mu}_n(f_k)^2\Big)\]
  for some constant $c > 0$ depending only on $K$.

  By Skorokhod's representation theorem we can assume that $\mu_n(f_1),\dots,\mu_n(f_m)$ and $\widetilde{\mu}_n(f_1),\dots,\widetilde{\mu}_n(f_m)$ converge almost surely. Note that by condition (3) of Lemma~\ref{lemma:auxiliary_function} the sequence $h(\sum_{k=1}^m \widetilde{\mu}_n(f_k)^2)$ is uniformly integrable. We have deduced the uniform (over $m$ and the functions $f_k$) estimate
  \[\E h\Big(\sum_{k=1}^m \mu(f_k)^2\Big) \le c \E h\Big(\sum_{k=1}^m \widetilde{\mu}(f_k)^2\Big).\]

  We next make a specific choice for the functions $f_k$. Given $\varepsilon >0$ choose a maximal set $x_1, x_2, \dots, x_m \in \T$ so that $|x_i - x_j| \ge \frac{\varepsilon}{2}$ for every $i \neq j$, where $m=m(\varepsilon)$. For each $1 \le k \le m$ define the continuous function $f_k$ by setting
  \[f_k(x) := \max\Big(0,\min\big(2 \frac{\varepsilon - d(x, x_k)}{\varepsilon}, 1 \big) \Big).\]
  Denote the diagonal of the product space $\T\times \T$ by $\Delta := \{(x,x) \in \T \times \T\}$ and its $2\varepsilon$-neighbourhood by $\Delta_{2\varepsilon} = \{(x,y) \in \T \times \T : d_{\T \times \T}((x,y), \Delta) < 2\varepsilon\}$. Here we use the metric $d_{\T \times \T}((x,y),(x',y')) := \max \{d(x,x'), d(y,y')\}$.
  Then for any measure $\lambda\in \M^+$ we have
  \begin{align*}
    (\lambda \otimes \lambda)(\Delta) & \le \sum_{k=1}^m \lambda(f_k)^2 \le \sum_{k=1}^m \lambda(B(x_k, \varepsilon))^2 \\
                          & \le N (\lambda \otimes \lambda)(\bigcup_{k=1}^m B(x_k, \varepsilon) \times B(x_k, \varepsilon)) \le N (\lambda \otimes \lambda)(\Delta_{2\varepsilon}),
  \end{align*}
 where $N > 0$ measures the maximal overlap of the balls $B(x_k, \varepsilon)$, and depends only on the doubling constant of the space $\T$.
  In particular for every $\varepsilon > 0$ we have
  \[\E h\big((\mu \otimes \mu)(\Delta)\big)] \le \E h\big(\sum_{k=1}^m \mu(f_k)^2\big)\le c\, \E h\big(\sum_{k=1}^m \widetilde{\mu}(f_k)^2\big) \le c N \E h\big((\widetilde{\mu} \otimes \widetilde{\mu})(\Delta_{2\varepsilon})\big).\]
  Letting $\varepsilon \to 0$ lets us conclude that $(\mu \otimes \mu)(\Delta) = 0$ almost surely, which entails that $\mu$ is non-atomic almost surely.
\end{proof}

\begin{remark}
One should note that the above proof is not valid as such if one just assumes that the dominance of the covariance is valid in one direction only. In a sense we perform both a convexity and a concavity argument while deriving the required inequality.
\end{remark}

\begin{lemma}\label{lemma:construction_of_Z}
    There exists a collection $Z_\varepsilon$ $(\varepsilon > 0)$ of H\"older-regular Gaussian
    fields on $\T$ such that $C_\varepsilon(x,y) := \E[Z_\varepsilon(x) Z_\varepsilon(y)]$ satisfies $C_\varepsilon(x,x)
    = 1$ for all $x \in \T$ and $\int_\T e^{Z_\varepsilon(x) - \frac{1}{2} \E[Z_\varepsilon{(x)}^2]} \, d\lambda(x)$ converges to $\lambda(\T)$ in
    $L^2(\Omega)$ for any non-atomic finite measure $\lambda \in \M^+(\T)$ as $\varepsilon \to 0$. Moreover, we have the bound
    \[\E\left|\int_\T e^{Z_\varepsilon(x) - \frac{1}{2} \E[Z_\varepsilon{(x)}^2]} \, d\lambda (x) - \lambda (\T)\right|^2 \le c (\lambda \otimes \lambda)
    (\{(x,y) \in \T : |x-y| < 2\varepsilon\})\]
    for some constant $c > 0$.
\end{lemma}
\begin{proof}
  Fix a sequence of independent standard Gaussian random variables $A_i$, $i \ge 1$.
  Let $\varepsilon > 0$ and choose a maximal set of points $a_1,\dots,a_n$ in $\T$ such that $|a_i - a_j| \ge
  \varepsilon /2$ for all $1 \le i < j \le n$. Let $B_i$ be the ball $B(a_i, \varepsilon)$. Then the balls $B_i$ cover
  $\T$ and we may form a Lipschitz partition of unity $p_1,\dots,p_n$ with respect to these balls. That is, $p_1,\dots,p_n$ are
  non-negative Lipschitz continuous functions such that $p_i(x) = 0$ when $x \notin B(a_i,\varepsilon)$ and for all $x \in \T$
  we have $\sum_{i=1}^n p_i(x) \equiv 1$.

  Define the field $Z_\varepsilon(x)$ by setting
  \[Z_\varepsilon(x) = \sum_{i=1}^n A_i \sqrt{p_i(x)},\]
  whence the covariance of $Z_\varepsilon$ is given by
  \[C_\varepsilon(x,y) := \E[Z_\varepsilon(x) Z_\varepsilon(y)] = \sum_{i=1}^n \sqrt{p_i(x) p_i(y)}.\]
  By the Cauchy-Schwartz inequality we see that
  \[C_\varepsilon(x,y) \le \sqrt{\sum_{i=1}^n p_i(x)}\sqrt{\sum_{i=1}^n p_i(y)} = 1\]
  for all $x,y \in \T$. Futhermore $C_\varepsilon(x,x) = 1$ for all $x \in \T$.

  Now a direct computation gives
  \[\E\left|\int_{\T} e^{Z_\varepsilon(x) - \frac{1}{2} \E[Z_\varepsilon{(x)}^2]} \, d\lambda (x) - \lambda (\T)\right|^2 = \int_{\T}
  \int_{\T} \big( e^{C_\varepsilon(x,y)} - 1 \big) \, d\lambda (x) \, d\lambda (y).\]
  Clearly when $|x-y| \ge 2\varepsilon$, we have $|x - a_i| + |y - a_i| \ge 2\varepsilon$, so one of $x$ or $y$ lies
  outside of $B_i$ for every $1 \le i \le n$, which implies that $C_\varepsilon(x,y) = 0$. Therefore we have
  \begin{align*}
      \int_{\T} \int_{\T} (e^{C_\varepsilon(x,y)} - 1) \, d\lambda (x) \, d\lambda(y) & = \int_{\{|x - y| < 2\varepsilon\}}
    \big( e^{C_\varepsilon(x,y)} - 1 \big) \, d(\lambda \otimes \lambda)(x,y) \\
    & \le(e - 1) (\lambda \otimes \lambda)(\{(x,y) : |x-y| < 2\varepsilon\}) ,
  \end{align*}
  and the right-hand side goes to $0$ as $\varepsilon \to 0$, since the non-atomicity of $\lambda$ guarantees that $(\lambda \otimes \lambda)(\{(x,x) : x \in \T\}) = 0$.
\end{proof}

\begin{proof}[Proof of Theorem~\ref{thm:uniqueness}]
    We will first assume that both sequences $(\mu_n)$ and $(\widetilde{\mu}_n)$ converge in distribution and show how to get rid of this
    condition at the end.

    Let $Z_\varepsilon$ be the independent field constructed as in Lemma~\ref{lemma:construction_of_Z}. Let $Y_\varepsilon(x) =
    \sqrt{K} Z_\varepsilon(x) + \varepsilon G$, where $G$ is an independent standard Gaussian random variable. A standard argument utilizing the continuity of the covariance of $Z_\varepsilon$ and compactness yields that for all large
    enough $n$ the covariance of the field $X_n+ Y_\varepsilon$ is greater than the covariance of the field
    $\widetilde{X}_n$ at every point $(x,y) \in \T \times \T$.

    We may assume, towards notational simplicity, that our probability space has the product form $\Omega = \Omega_1 \times \Omega_2$, and for $(\omega_1,\omega_2) \in \Omega$ one has
    $X_n((\omega_1,\omega_2)) = X_n(\omega_1)$ and $\widetilde{X}_n((\omega_1,\omega_2)) = \widetilde{X}_n(\omega_1)$ together with $Y_\varepsilon((\omega_1,\omega_2)) = Y_\varepsilon(\omega_2)$ for all $\varepsilon > 0$.
    Let $\varphi \colon \left[0,\infty\right) \to \left[0,\infty\right)$ be a bounded, continuous and concave function. Then by Kahane's convexity inequality we have
    \begin{align*}
       \E\left[\varphi\left(\int_T f(x) e^{X_n(x) + Y_\varepsilon(x) - \frac{1}{2} \E[X_n(x)^2] - \frac{1}{2}
       \E[Y_\varepsilon(x)^2]} \, d\rho_n(x))\right)\right] & \le \\
                    \E\left[\varphi\left(\int_\T f(x) e^{\widetilde{X}_n -
                    \frac{1}{2} \E[\widetilde{X}_n(x)^2]} \, d\rho_n(x)\right)\right]
    \end{align*}
    for all non-negative $f \in C(\T)$. Since for all fixed $\omega_2 \in \Omega_2$, $Y_\varepsilon(\omega_2)(x)-\frac{1}{2}
        \E[Y_\varepsilon(x)^2]$ is a continuous function on $\T$, we see that
    \begin{align*}
        \E_{\Omega_1}\left[\varphi\left(\int_\T f(x) e^{X_n(x) + Y_\varepsilon(x) - \frac{1}{2} \E[X_n(x)^2] - \frac{1}{2}
        \E[Y_\varepsilon(x)^2]} \, d\rho_n(x)\right)\right] & \to \\
                    \E_{\Omega_1}\left[\varphi\left(\int_\T f(x) e^{Y_\varepsilon(x) - \frac{1}{2}
                    \E[Y_\varepsilon(x)^2]} \, d\mu(x) \right)\right]
    \end{align*}
    as $n \to \infty$.
    In particular we have by Fatou's lemma that
    \begin{align}\label{eq:fatou}
       & \E\left[\varphi\left(\int_\T f(x) e^{Y_\varepsilon(x) - \frac{1}{2} \E[Y_\varepsilon(x)^2]} \,
        d\mu(x)\right)\right] \\
       =& \E_{\Omega_2} \lim_{n \to \infty} \E_{\Omega_1} \left[\varphi\left(\int_\T f(x)
      e^{X_n(x) + Y_\varepsilon(x) - \frac{1}{2} \E[X_n(x)^2] - \frac{1}{2} \E[Y_\varepsilon(x)^2]} \, d\rho_n(x) \right)\right] &\nonumber \\
       \le&\; \liminf_{n \to \infty} \E\left[\varphi\left(\int_\T f(x) e^{\widetilde{X}_n - \frac{1}{2}
        \E[\widetilde{X}_n(x)^2]} \, d\rho_n(x)\right)\right] \nonumber\\
             =&\;       \E\left[\varphi\left(\int_\T f(x) \, d\widetilde{\mu}(x)\right)\right].\nonumber
    \end{align}
    Accoding to Lemma~\ref{lemma:construction_of_Z}, for almost every $\omega_1 \in \Omega_1$ we know that
    \begin{equation}\label{eq:convergence_g}
    g_\varepsilon:= \int_\T f(x) e^{Y_\varepsilon(x) - \frac{1}{2}
    \E[Y_\varepsilon(x)^2]} \, d\mu(x) \;\; \underset{\varepsilon\to 0}\longrightarrow \;\;g:=\int_\T f(x) \, d\mu(x)
  \end{equation}
  in $L^2(\Omega_2)$. We next note that for a suitable fixed sequence $\varepsilon_k \to 0$ this convergence also happens for almost every $\omega_2 \in \Omega_2$. By Lemma~\ref{lemma:construction_of_Z} we have the estimate
  \[\|g_\varepsilon - g\|^2_{L^2(\Omega_2)} \le c \|f\|_{C(\T)}^2 (\mu \otimes \mu) (\{|x-y| < 2\varepsilon\}),\]
     where $c > 0$ is some constant. Choose the sequence $\varepsilon_k$ so that
     \[\P[c \|f\|_{C(\T)}^2 (\mu \otimes \mu)(\{|x-y| < 2\varepsilon_k\}) > 4^{-k}] \le \frac{1}{k^2},\]
     which is possible because $(\mu \otimes \mu)(\{(x,x) : x \in \T\}) = 0$ almost surely.  By the Borel--Cantelli lemma there exists a random index $k_0(\omega_1) \ge 1$ such that with probability $1$ we have
     \[\|g_{\varepsilon_k} - g\|^2_{L^2(\Omega_2)} \le c \|f\|_{C(\T)}^2 (\mu \otimes \mu)(\{|x-y| < 2\varepsilon_k\}) \le 4^{-k}\]
     for all $k \ge k_0(\omega_1)$. Now a standard argument verifies the almost sure convergence in \eqref{eq:convergence_g}.

    The almost sure convergence finally lets us to conclude for all non-negative $f\in C(T)$ and non-negative, bounded, continuous and concave $\varphi$ that
    \[\E\left[\varphi\left(\int_\T f(x) \, d\mu(x)\right)\right] \le \E\left[\varphi\left(\int_\T f(x) \,
    d\widetilde{\mu}(x)\right)\right].\]
    Similar inequality also holds with the measures $\mu$ and $\widetilde{\mu}$ switched, so we actually have
    \[\E\left[\varphi\left(\int_\T f(x) \, d\mu(x)\right)\right] = \E\left[\varphi\left(\int_\T f(x) \,
    d\widetilde{\mu}(x)\right)\right].\]
    It is well known that this implies $\mu \sim \widetilde{\mu}$.

    Let us now finally observe that one can drop the assumption that both families of measures converge. By
    Lemma~\ref{lemma:convergence} and Prokhorov's theorem we know that every subsequence $\mu_{n_k}$ has a further subsequence
    that converges in distribution to a random measure. Lemma~\ref{lemma:diagonal_approximation} ensures that the limit
    measure of any converging sequence has almost surely no atoms, and hence by the previous part of the proof
    this limit must equal $\widetilde{\mu}$. This implies that the original sequence must converge to $\widetilde{\mu}$ as well.
\end{proof}

\begin{remark}\label{rmk:loosen}
  Our proof of Theorem~\ref{thm:uniqueness} may be modified in a way that allows the conditions \eqref{ehto1} and \eqref{ehto2} to be somewhat relaxed. E.g.\ in the case of subcritical logarithmically correlated fields it is basically enough to have for $\varepsilon > 0$ the inequality
  \[|C_n(s,t) - \widetilde{C}_n(s,t)| \le \varepsilon(1 + \log^+ \frac{1}{|s-t|})\]
  for $n \ge n(\varepsilon)$. Analogous results exist also for the critical chaos, but in this case the specific conditions are heavily influenced by the approximation sequence $X_n$ one uses.
\end{remark}

\section{Convergence in probability}\label{sec:inprobability}

\noindent In the previous section convergence was established in distribution, which often suffices, and the main focus was on the uniqueness of the limit. In the present section we establish the convergence also in probability, assuming that this is true for the comparison sequence $\widetilde{\mu}_n$, which is constructed using approximating sequence $(\widetilde{X}_n)$ that has independent increments. Convergence in probability in the subcritical case was also discussed in \cite{shamov2014gaussian}, and our Theorem~\ref{thm:convergence_in_probability} below can be seen as an alternative way to approach the question. For the proof we need the following two auxiliary observations.

\begin{lemma}\label{lemma:convergence_in_distribution_implies_convergence_in_probability}
  Let ${\mathcal F}_1\subset {\mathcal F}_2\subset\ldots$ be an increasing sequence of sigma-algebras and denote ${\mathcal F}_\infty:=\sigma(\bigcup_{j=1}^\infty{\mathcal F}_k) \subset \mathcal{F}$.
  Assume that the real random variables $X, X_1,X_2,\ldots$ satisfy: $X$ is ${\mathcal F}_\infty$-measurable, and for any ${\mathcal F}_j$ measurable set $E$ {\rm (}with arbitrary $j\geq 1${\rm )} it holds that
\begin{equation}\label{eq:ehto}
\chi_EX_k\overset{d}{\longrightarrow} \chi_EX\qquad \textrm{as}\quad k\to \infty.
\end{equation}
Then $X_k\overset{P}{\longrightarrow} X$ as $k\to\infty$.
\end{lemma}
\begin{proof}
We first verify that \eqref{eq:ehto} remains true also if the set $E$ is just ${\mathcal F}_\infty$-measurable.
For that end define $h_j := \E(\chi_E|{\mathcal F}_j)$ and construct an ${\mathcal F}_\infty$-measurable approximation $E_j:=h_j^{-1}((1/2,1])$. The martingale convergence theorem yields that $\P (E_j\Delta E)\to 0$ as $j\to\infty$. Since the claim holds for each $E_j$, it also follows for the set $E$ by a standard approximation argument.

Let us then establish the stated convergence in probability.
Fix $\varepsilon >0$ and pick $M>0$ large enough so that $\P (|X|>M/2)\leq \varepsilon/2$, and such that $\P (|X|=M)=0$.
Then for some $k_0$ we have that $\P (|X_k|\geq M)\leq \varepsilon$ if $k\geq k_0$.
Divide the interval $(-M,M]$ into non overlapping half open intervals $I_1,\ldots , I_\ell$ of length less than $\varepsilon/2$ and denote $E_j:=X^{-1}(I_j)$ for $j=1,\ldots , \ell$.
In the construction we may assume that $0$ is the center point of one of these intervals and $\P(X=a)=0$ if $a$ is an endpoint of any of the intervals.
We fix $j$ and apply condition \eqref{eq:ehto}
to deduce that $\chi_{E_j}X_k\overset{d}{\longrightarrow} \chi_EX$ as $k\to \infty$.
Assume first that $0\not\in I_j$.
Then the Portmonteau theorem yields that $\lim_{k\to\infty}\P (\chi_{E_j}X_k\in I_j)=\P (\chi_{E_j}X\in I_j)$, or in other words
\[\P (\{X\in I_j\}\cap\{X_k\in I_j\})\to \P(X\in I_j)\quad\textrm{as}\quad k\to \infty.\]

In particular, for large enough $k$ we have that
\begin{equation}\label{eq:ehto2}
  \P\big(E_j\cap (|X-X_k|>\varepsilon)\big)\leq \frac{\varepsilon}{2\ell}
\end{equation}
If $0\in I_j$ we obtain in a similar vein that $\lim_{k\to\infty}\P(\chi_{E_j}X_k\in (I_j)^c)=\P (\chi_{E_j}X\in (I_j)^c)=0$, or in other words
$\P (\{X\in I_j\} \cap \{X_k\in I_j^c\})\to 0,$ so that we again get that $\P\big(E_j\cap (|X-X_k|>\varepsilon\big)\leq \frac{\varepsilon}{2\ell}$ for large enough $k$. By summing the obtained inequalities for $j=1,\ldots ,\ell$ and observing that $\P(\bigcup_{k=1}^\ell E_k)>1-\varepsilon/2$ we deduce for large enough $k$ the inequality $\P(|X-X_k|>\varepsilon)<\varepsilon$, as desired.
\end{proof}

\begin{lemma}\label{lemma:measure_equal_in_prob}
  Let $X$ be a H\"older-regular Gaussian field on $\T$ that is independent of the random measures $\mu$ and $\nu$ on $\T$.

  \begin{enumerate}[label={\bf (\roman*)}]
    \item If $e^X \mu \sim e^X \nu$, then also $\mu \sim \nu$.

    \item If $(\mu_n)$ is a sequence of random measures such that the sequence $(e^X \mu_n)$ converges in distribution, then also the sequence $(\mu_n)$ converges in distribution.
  \end{enumerate}
\end{lemma}

\begin{proof}
  We will first show that if $X$ is of the simple form $Nf$ with $N$ a standard Gaussian random variable and $f \in C(\T)$, then the claim holds. To this end let us fix $g \in C(\T)$ and consider the function $\varphi \colon \reals \to \complexes$ defined by
  \[\varphi(x) = \E[\exp\big(i \int e^{Nf} e^{-xf} g \, d\mu\big)] = \E[\exp\big(i \int e^{Nf} e^{-xf} g \, d\nu\big)].\]
  Because $N$ is independent of $\mu$ and $\nu$, we may write
  \[\varphi(x) = \int_{-\infty}^\infty \E[\exp(i \int e^{(y-x)f} g \, d\mu)] \frac{1}{\sqrt{2\pi}} e^{-\frac{y^2}{2}} \, dy.\]
  By denoting $u(t) = \E[\exp(i \int e^{-tf} g \, d\mu)]$, $v(t) = \E[\exp(i \int e^{-tf} g \, d\nu)]$ and $h(x) = \frac{1}{\sqrt{2\pi}} e^{-\frac{x^2}{2}}$, we see that $\varphi(x) = (u * h)(x) = (v * h)(x)$. Because the Fourier transform of $h$ is also Gaussian we deduce by taking convolutions that the Fourier transforms $\widehat{u}$ and $\widehat{v}$ coincide as Schwartz distributions. Since $u$ and $v$ are continuous, this implies that $u(x) = v(x)$ for all $x$. In particular setting $x=0$ gives us
  \[\E[\exp(i \int g \, d\mu)] = \E[\exp(i \int g \, d\nu)],\]
  for all $g \in C(\T)$, whence the measures $\mu$ and $\nu$ have the same distribution.

  To deduce the general case, note that we have the Karhunen--Lo\`eve decomposition
  \[X = \sum_{k=1}^\infty N_k f_k\]
  where $N_k$ are standard Gaussian random variables and $f_k \in C(\T)$ for all $k \in \naturalnumbers$. Moreover the above series converges almost surely uniformly. (See for example \cite[Theorem 3.1.2.]{adler2009random}.) By the first part of the proof we know that $e^{\sum_{k=n}^\infty N_k f_k} \, \mu$ and $e^{\sum_{k=n}^\infty N_k f_k} \, \nu$ have the same distribution for all $n \in \naturalnumbers$. By the dominated convergence theorem we have
  \begin{align*}
    \E[\exp(i \int g \, d\mu)] & = \lim_{n\to\infty} \E[\exp(i \int e^{\sum_{k=n}^\infty N_k f_k} g \, d\mu)] \\
                               & = \lim_{n\to\infty} \E[\exp(i \int e^{\sum_{k=n}^\infty N_k f_k} g \, d\nu)] = \E[\exp(i \int g \, d\nu)]
  \end{align*}
  for all $g \in C(\T)$, which shows the claim.

  The second part of the lemma follows from the first part. Since $\sup_{t \in \T} X(t) < \infty$ almost surely, one checks that the sequence $(\mu_n)$ inherits the tightness of the sequence $(e^X \mu_n)$. It is therefore enough to show that any two converging subsequences have the same limit. Indeed, assume that $\mu_{k_j} \to \mu$ and $\mu_{n_j} \to \nu$ in distribution. Then by independence we have $e^X \mu_{k_j} \to e^X \mu$ and $e^X \mu_{n_j} \to e^X \nu$, but by assumption the limits are equally distributed and hence also $\mu$ and $\nu$ have the same distribution.
\end{proof}

A typical example of a linear regularization process described in the following definition is given by a standard convolution approximation sequence. We denote by $C^\alpha(\T)$ the Banach space of $\alpha$-H\"older continuous functions on $\T$.

\begin{definition}\label{def:linear_reg}
  Let $(X_k)$ be a sequence of approximating fields on $\T$.
  We say that a sequence $(R_n)$ of linear operators $R_n \colon \bigcup_{\alpha \in (0,1)} C^\alpha(\T) \to C(\T)$ is a \emph{linear regularization process} for the sequence $(X_k)$ if the following properties are satisfied:
  \begin{enumerate}
    \item We have $\lim_{n \to \infty} \|R_n f - f\|_\infty = 0$ for all $f \in \bigcup_{\alpha \in (0,1)} C^\alpha(\T)$.
    \item The limit $R_n X := \lim_{k \to \infty} R_n X_k$ exists in $C(\T)$ almost surely.
  \end{enumerate}
\end{definition}

\begin{theorem}\label{thm:convergence_in_probability}
  Assume that the increments $\{X_{m+1} - X_{m} : m \ge 1\}$ of the approximating fields $X_m$ are independent
  and that there is the convergence in probability
  \begin{equation}\label{eq:lyhyt}
  d\widetilde{\mu}_n := e^{X_n - \frac{1}{2}\E[X_n^2]} \, d\rho_n \underset{n \to \infty}{\overset{P}{\longrightarrow}} \widetilde{\mu}.
  \end{equation}
  Let $R_n$ be some linear regularization process for the sequence $X_k$ such that
  \[e^{R_n X - \frac{1}{2} \E[(R_n X)^2]} \, d\rho_n \underset{n \to \infty}{\overset{d}{\longrightarrow}} \widetilde{\mu}.\]
  Then also $d\mu_n = e^{R_n X - \frac{1}{2} \E[(R_n X)^2]} \, d\rho_{n}$ converges to $\widetilde{\mu}$ in probability.
\end{theorem}

\begin{remark}
  As in Remark~\ref{rmk:uniqueness_non_compact} the above theorem extends to the case of a non-compact $\T$ when the assumptions are suitably reinterpreted. In a particular application it is also enough to assume the condition (1) in Definition~\ref{def:linear_reg} for one suitable fixed value of $\alpha > 0$, if the exponent of the H\"older regularity of the approximating fields is known.
\end{remark}

\begin{proof}
  Define the filtration $\mathcal{F}_n := \sigma(X_1, \dots, X_n)$.
  First of all, since $e^{X_n - \frac{1}{2} \E[X_n^2]} \, d\rho_n$ converges to $\widetilde{\mu}$ in probability as $n\to\infty$, we also have
  \[e^{X_n - X_k - \frac{1}{2} \E[(X_n - X_k)^2]} \, d\rho_n \underset{n \to \infty}{\overset{P}{\longrightarrow}} e^{-X_k + \frac{1}{2} \E[X_k^2]} \widetilde{\mu} \quad \text{for every } k \ge 1.\]
  To see this, one uses that $\E[(X_n - X_k)^2] = \E[X_n^2] - \E[X_k^2]$ and considers almost surely converging subsequences, if necessary.
  We denote $\eta_k := e^{-X_k + \frac{1}{2} \E[X_k^2]} \widetilde{\mu}$.

  Notice that $\E[(R_n X)(R_n X_k)] = \E[(R_n X_k)^2]$ by the independent increments and the definition of $R_n X$. We may thus write
  \begin{align}\label{eq:decomposition_at_level_k}
    d\mu_n& =
  e^{R_n X - \frac{1}{2} \E[(R_n X)^2]}\, d\rho_n \\& = \Big[ e^{R_n X_k - X_k+\frac{1}{2} \E[X_k^2-(R_n X_k)^2]}\Big]e^{X_k - \frac{1}{2} \E[ X_k^2]}e^{R_n (X - X_k) - \frac{1}{2} \E[(R_n (X - X_k))^2]} \, d\rho_n.\nonumber
  \end{align}
  Above on the right hand side the term in brackets is negligible as $n\to\infty$. To see this, we note first that  $e^{R_n X_k - X_k}$ tends almost surely to the constant function $1$ uniformly according to  Definition~\ref{def:linear_reg}(1). Moreover, $\E[X_k^2-(R_n X_k)^2]$ tends to $0$ in $C(\T)$, since the field $X_k$ takes values in a fixed $C^\gamma(\T)$ for some $\gamma > 0$, and by the Banach--Steinhaus theorem
  $\sup_{n \ge 1} \|R_n\|_{C^\gamma(T) \to C(\T)} < \infty$.
 Namely,
  \begin{align*}
    \|\E[ X_k^2-(R_n X_k)^2  ]\|_{C(\T)} & \le \E \| (X_k-R_n X_k)( X_k + R_nX_k)\|_{C(\T)} \\
                                        & \le \E \Big[\|X_k-R_n X_k\|_{C(\T)} \|X_k+R_n X_k\|_{C(\T)}\Big] \\
                                        & \lesssim \E \|X_k\|_{C^\gamma(\T)}^2,
  \end{align*}
 whence the dominated convergence theorem applies, since $\|X_k\|_{C^\gamma(\T)}$ has a super exponential tail by Fernique's theorem.
 All in all, invoking the assumption on the convergence of $\widetilde{\mu}_n$ we deduce that
  \begin{equation}\label{eq:nu_convergence}
    e^{X_k - \frac{1}{2} \E[X_k^2]} e^{R_n (X - X_k) - \frac{1}{2} \E[(R_n (X - X_k))^2]} \, d\rho_n\longrightarrow \widetilde{\mu}
  \end{equation}
  in distribution as $n\to\infty$.

  By Lemma~\ref{lemma:measure_equal_in_prob} we thus have the distributional convergence
  \[e^{R_n (X - X_k) - \frac{1}{2} \E[(R_n (X - X_k))^2]} \, d\rho_n\longrightarrow \nu_k\quad\textrm{as}\;\;n\to\infty,
  \]
  where the limit $\nu_k$ may be  assumed to be independent of $\mathcal{F}_k$.
  In particular, recalling \eqref{eq:nu_convergence} we deduce that $e^{X_k - \frac{1}{2} \E[X_k^2]} \nu_k$ has the same distribution as $\widetilde{\mu} = e^{X_k - \frac{1}{2} \E[X_k^2]} \eta_k$.
Lemma~\ref{lemma:measure_equal_in_prob} now verifies that $\nu_k\sim \eta_k$.
  In order to invoke Lemma~\ref{lemma:convergence_in_distribution_implies_convergence_in_probability}, fix  any $\mathcal{F}_k$ measurable bounded random variable $g$.
  Then $g$ and $X_k$ are independent of $X - X_k$, and we therefore  have the distributional convergence
  \begin{align}\label{eq5}
  g e^{X_k - \frac{1}{2} \E[X_k^2]} &e^{R_n (X - X_k) - \frac{1}{2} \E[(R_n (X - X_k))^2]} \, d\rho_n\\
 \underset{n\to\infty} \longrightarrow \;\;&g e^{X_k - \frac{1}{2} \E[X_k^2]} \, d\nu_k \;\; \sim \;\; g e^{X_k - \frac{1}{2} \E[X_k^2]} \, d\eta_k
 \; = \; g \, d\widetilde{\mu},\nonumber
  \end{align}
where the second last equality followed  by independence.
  Finally, again by the negligibility of the term $e^{R_n X_k - X_k}e^{-\frac{1}{2} \E[X_k^2-(R_n X_k)^2]}$ and using \eqref{eq:decomposition_at_level_k} we see that \eqref{eq5} in fact entails the convergence of $g \, d\mu_n$ to $g \, d\widetilde{\mu}$ in distribution.
  At this stage Lemma~\ref{lemma:convergence_in_distribution_implies_convergence_in_probability} applies and the desired claim follows.
\end{proof}

\begin{remark}
  In the previous theorem it was crucial that we already have an approximating sequence of fields along which the corresponding chaos converges in probability.
  In general if one only assumes convergence in distribution in $\eqref{eq:lyhyt}$, one may not automatically expect that it is possible to lift the convergence to that in probability, even for natural approximating fields. However, for most of the standard constructions of subcritical chaos this problem does not occur, as we have even almost sure convergence in \eqref{eq:lyhyt} due to the martingale convergence theorem.
\end{remark}

\section{Two auxiliary results}\label{sec:auxiliary}

\noindent In this section we provide a couple of useful auxiliary tools dealing with convolution approximations and convergence of perturbed chaos.

The next lemma and its corollaries show that any two convolution approximations (with some regularity) applied to log-normal chaos stay close to each other in the sense of Theorem~\ref{thm:uniqueness}.

\begin{lemma}\label{lemma:bmo_convolutions}
  Let $\varphi, \psi \colon \reals^d \to \reals$ satisfy $\int \varphi(x) \, dx = \int \psi(x) \, dx = 1$ and $|\varphi(x)|, |\psi(x)| \le C (1 + |x|)^{-(d + \delta)}$ for all $x \in \reals^d$ with some constants $C,\delta > 0$. Then if $u \in BMO(\reals^d)$, we have
  \[|(\varphi_\varepsilon * u)(x) - (\psi_\varepsilon * u)(x)| \le K\]
  for some constant $K > 0$ not depending on $\varepsilon$.
\end{lemma}

\begin{proof}
  One can use the mean zero property and decay of $\varphi - \psi$ together with a standard BMO-type estimate \cite[Proposition 7.1.5.]{grafakos2009modern} to see that for any $\varepsilon > 0$ we have
  \begin{align*}
    & \left|\int_{\reals^d} (\varphi_\varepsilon - \psi_\varepsilon)(t) u(x-t) \, dt\right| = \left|\int_{\reals^d} (\varphi - \psi)(t) \Big(u(\varepsilon (x-t)) - \fint_{B(0,1)} u(\varepsilon (x-s))\,ds\Big) \, dt\right| \\
    & \le \int_{\reals^d} \frac{|u(\varepsilon (x-t)) - \fint_{B(0,1)} u(\varepsilon (x-s)) \, ds|}{(1 + |t|)^{d + \delta}} \, dt \\
    & \le  C_{d,\delta} \|u(\varepsilon (x - \cdot))\|_{BMO} = C_{d,\delta} \|u\|_{BMO}.\qedhere
  \end{align*}
\end{proof}

\begin{corollary}\label{cor:convolutions}
  Let $f(x,y) = 2d\beta^2 \log^+ \frac{1}{|x-y|} + g(x,y)$ be a covariance kernel of a distribution valued field $X$ defined on $\reals^d$. Here $g$ is a bounded uniformly continuous function. Assume that $\varphi$ and $\psi$ are two locally H\"older continuous convolution kernels in $\reals^d$ that satisfy the conditions of Lemma~\ref{lemma:bmo_convolutions}. Let $(\varepsilon_n)$ be a sequence of positive numbers $\varepsilon_n$ converging to $0$. Then the approximating fields $X_n := \varphi_{\varepsilon_n} * X$ and $\widetilde{X}_n := \psi_{\varepsilon_n} * X$ satisfy the conditions \eqref{ehto1} and \eqref{ehto2} of Theorem~\ref{thm:uniqueness}.
\end{corollary}

\begin{proof}
  The function $\ell(x) := 2d\beta^2 \log^+ \frac{1}{|x|}$ belongs to $BMO(\reals^d)$ since $\log |x| \in BMO(\reals^d)$, see for example \cite[Example 7.1.3]{grafakos2009modern}. One computes that the covariance of $\varphi_{\varepsilon} * X$ equals
  \[\int \int \varphi_{\varepsilon}(x-t) \varphi_{\varepsilon}(y-s) \ell(t-s) \, dt \, ds + \int \int \varphi_{\varepsilon}(x-t) \varphi_{\varepsilon}(y-s) g(t,s) \, dt \, ds.\]
  Because $g$ is bounded and uniformly continuous the second term goes to $g(x,y)$ uniformly, so we may without loss of generality assume that $g(x,y) = 0$. The first term equals $(\varphi_{\varepsilon} * \varphi_{\varepsilon}(- \cdot) * \ell)(x-y)$, so the condition \eqref{ehto1} follows from Lemma~\ref{lemma:bmo_convolutions} applied to the convolution kernels $\varphi * \varphi(-\cdot)$ and $\psi * \psi(-\cdot)$. Here one easily checks that also $\varphi * \varphi(-\cdot)$ satisfies the conditions of Lemma~\ref{lemma:bmo_convolutions} and that $(\varphi * \varphi(-\cdot))_\varepsilon = \varphi_\varepsilon * \varphi_\varepsilon(-\cdot)$. Finally, the condition \eqref{ehto2} is immediate.
\end{proof}

\begin{remark}
  One may easily state localized versions of the above corollary.
\end{remark}

\begin{corollary}\label{cor:convolutions_circle}
  Assume that $f(x,y) = 2\beta^2 \log^+ \frac{1}{2|\sin(\pi (x-y))|} + g(x,y)$ is the covariance of a {\rm (}distribution valued{\rm )} field $X$ on the unit circle. Here $g$ is a bounded continuous function that is $1$-periodic in both variables $x$ and $y$ and we have identified the unit circle with $\reals/\integers$. Assume that $\varphi$ and $\psi$ are two locally H\"older continuous convolution kernels in $\reals$ that satisfy the conditions of Lemma~\ref{lemma:bmo_convolutions}, and let $(\varepsilon_n)$ be a sequence of positive numbers $\varepsilon_n$ converging to $0$. Then the approximating fields $X_n := \varphi_{\varepsilon_n} * X$ and $\widetilde{X}_n := \psi_{\varepsilon_n} * X$ satisfy the conditions \eqref{ehto1} and \eqref{ehto2} of Theorem~\ref{thm:uniqueness}.
\end{corollary}

\begin{remark}
  Above when defining the approximating fields $X_n$ we assume that $X$ stands for the corresponding periodized field on $\reals$ and the fields $X_n$ will then automatically be periodic so that they also define fields on the unit circle.
\end{remark}

\begin{proof}
  One easily checks that $\ell(x) = 2\beta^2 \log^+ \frac{1}{2|\sin(\pi x)|}$ is in $BMO(\reals)$.
  The rest of the proof is analogous to the one of the previous corollary.
\end{proof}

The Proposition~\ref{prop:adding_regular_field} is needed later on in a localization procedure that is used to carry results from the real line to the unit circle. For its proof we need the following lemma.

\begin{lemma}\label{lemma:random_function_lp_convergence}
  Assume that $\mu_n$ is a sequence of random measures that converges to $\mu$ weakly in $L^p(\Omega)$.
  Let $F \colon \Omega \to C(\T)$ be a function valued random variable and assume that there exists $q > 0$ such that
  \[\E\left|\sup_{x \in \T} F(x)\right|^\alpha < \infty\]
  for some $\alpha > \frac{pq}{p-q}$.
  Then $\int F(x) \, d\mu_n(x)$ tends to $\int F(x) \, d\mu(x)$ in $L^q(\Omega)$.
\end{lemma}

\begin{proof}
  It is again enough to show that any subsequence possesses a converging subsequence with the right limit. To simplify notation let us denote by $\mu_n$ an arbitrary subsequence of the original sequence.

  Directly from the definition of the metric in the space $\M^+$ we see that $\mu_n \to \mu$ in probability, meaning that we can pick a subsequence $\mu_{n_j}$ that converges almost surely. Then the almost sure convergence holds also for the sequence $\int F(x) \, d\mu_{n_j}(x)$.
   Finally, for any allowed value of $q$ a standard application of H\"older's inequality shows that $\E |\int F(x) \, d\mu_{n_j}(x)|^{q+\varepsilon}$ is uniformly bounded for some $\varepsilon > 0$. This yields uniform integrability and we may conclude.
\end{proof}

\begin{proposition}\label{prop:adding_regular_field}
  Let $(X_n)$ and $(Z_n)$ be two sequences of {\rm (}jointly Gaussian{\rm )} H\"older\-regular Gaussian fields on $\T$.
  Assume that the pseudometrics arising in Definition~\ref{def:holder_regular} can be chosen to have the same H\"older exponent and constant for all the fields $Z_n$. Assume further that there exists a H\"older-regular Gaussian field $Z$ such that $Z_n$ converges to $Z$ uniformly almost surely and that $\E[X_n(x) Z_n(x)]$ converges uniformly to some bounded continuous function $x \mapsto \E[X(x) Z(x)]$.
  Then if the measures
  \[d\mu_n(x) := e^{X_n(x) - \frac{1}{2}\E[X_n(x)^2]} \, d\rho_n(x)\]
  converge weakly in $L^p(\Omega)$ to a measure $\mu$, also the measures
  \begin{align*}
  d\nu_n(x) & := e^{(X_n(x) + Z_n(x)) - \frac{1}{2}\E[(X_n(x) + Z_n(x))^2]} \, d\rho_n(x) \\
            & = e^{Z_n(x) - \frac{1}{2}\E[Z_n(x)^2] - \E[X_n(x) Z_n(x)]} \, d\mu_n(x)
  \end{align*}
  converge weakly in $L^q(\Omega)$ for all $q < p$ to the measure
  \[d\nu(x) := e^{Z(x) - \frac{1}{2}\E[Z(x)^2] - \E[X(x) Z(x)]} \, d\mu(x).\]
\end{proposition}

\begin{proof}
  By a standard application of the Borell--TIS inequality \cite[Theorem 4.1.2]{adler2009random} we have the following uniform bound
\begin{equation}\label{eq:Z_bound}
\E e^{r \sup_{x \in \T} Z_n(x)} \le C_r
\end{equation}
  for all $r > 0$.
  Fix $\varepsilon > 0$ and for all $n \ge 1$ define
  \[A_n^{\varepsilon} := \{\omega \in \Omega : \sup_{x \in \T} |Z_k(x) - Z(x)| < \varepsilon \text{ for all } k \ge n\}.\]
  By the assumption on uniform convergence we have $\P[A_n^{\varepsilon}] \to 1$ as $n \to \infty$.

  Fix $f \in C(\T)$, which we may assume to be non-negative, and let $0 < q < p$.
  We first show that
\[\E[\chi_{\Omega \setminus A_n^{\varepsilon}} |\nu_n(f) - \nu(f)|^q] \to 0\]
as $n \to \infty$.
It is enough to verify uniform integrability by checking that
\begin{equation}\label{eq:nu_n_q_unif}
\sup_{n \ge 1} \E|\nu_n(f)|^{p'} + \E|\nu(f)|^{p'} < \infty
\end{equation}
for some $q < p' < p$. This in turn follows easily from the assumed uniform $L^p$ bound for $\mu_n$ by using H\"older's inequality together with \eqref{eq:Z_bound}.

To handle the remaining term $\E[\chi_{A_n^{\varepsilon}} |\nu_n(f) - \nu(f)|^q]$ we use the defining property of the set $A_n^{\varepsilon}$, i.e.\ $|Z_n(x) - Z(x)| < \varepsilon$ for all $x \in \T$.
By choosing $n$ large enough and by using \eqref{eq:Z_bound} we may further assume that $\sup_{x \in \T} |\E[Z_n(x)^2] - \E[Z(x)^2]| < \varepsilon$ and $\sup_{x \in \T} |\E[Z_n(x) X_n(x)] - \E[Z(x) X(x)]| < \varepsilon$.
  It follows that when $\omega \in A_n^\varepsilon$, we have
  \[e^{-3\varepsilon} c_n(f) \le \nu_n(f) \le e^{3\varepsilon} c_n(f),\]
  where
  \[c_n(f) = \int f(x) e^{Z(x) - \frac{1}{2} \E[Z(x)^2] - \E[Z(x) X(x)]} \, d\mu_n(x).\]
  By combining this with the bound \eqref{eq:nu_n_q_unif} we see that $\E |\nu_n(f) - c_n(f)|^q \to 0$ as $\varepsilon \to 0$, uniformly in $n$.
  Finally, by Lemma~\ref{lemma:random_function_lp_convergence} we have $c_n(f) \to \nu(f)$ in $L^q(\Omega)$. This finishes the proof.
\end{proof}

\section{An application (Proof of Theorem~\ref{co:1})}\label{sec:application}

\noindent The main purpose of this chapter is to prove Theorem~\ref{co:1} and explain carefully the approximations mentioned there. For the reader's convenience we try to be fairly detailed, although some parts of the material are certainly well-known to the experts.

We start by defining the approximation $X_{2,n}$ of the restriction of the free field on the unit circle $S^1 := \{(x_1,x_2) \in \reals^2 : x_1^2 + x_2^2 = 1\}$, or if needed $S^1 := \{z \in \complexes : |z| = 1\}$ as we freely identify $\reals^2$ with $\complexes$. Following \cite{astala2011random} recall that the trace of the Gaussian free field on the unit circle is defined to be the Gaussian field\footnote{Observe that we have in fact multiplied the standard definition by $\sqrt{2}$ to get the critical field. Also the innocent constant term $2 \sqrt{\log 2} G$ is often omitted in the definition.}
\begin{equation}\label{eq:trace_of_gff}
  X(x) = 2 \sqrt{\log 2}G + \sqrt{2}\sum_{k=1}^\infty \Big(\frac{A_k}{\sqrt{k}} \cos(2\pi k x) + \frac{B_k}{\sqrt{k}} \sin(2\pi k x)\Big),
\end{equation}
where $A_k, B_k$ and $G$ are independent standard Gaussian random variables.
The field $X$ is distribution valued and its covariance (more exactly, the kernel of the covariance operator) can be calculated to be
\begin{equation}\label{eq:covariance_gff}
\E[X(x)X(y)] = 4 \log(2) + 2 \log \frac{1}{2|\sin(\pi (x - y))|}.
\end{equation}
A natural approximation of $X$ is then obtained by considering the partial sum of the Fourier series
\[X_{2,n}(x) := 2 \sqrt{\log 2} G + \sqrt{2} \sum_{k=1}^n \Big(\frac{A_k}{\sqrt{k}} \cos(2\pi k x) + \frac{B_k}{\sqrt{k}} \sin(2\pi k x)\Big).\]

Another way to get hold of this covariance is via the periodic upper half-plane white noise expansion that we define next -- recall that the non-periodic hyperbolic white noise $\W$ and the hyperbolic area measure $\mhyp$ were already defined in the introduction.
We define the periodic white noise $\Wper$ to be
\[\Wper(A) = \W(A \bmod 1),\]
where $A \bmod 1 = \{(x \bmod 1,y) : (x,y) \in A\}$ and we define $x \bmod 1$ to be the number $x' \in [-\frac{1}{2},\frac{1}{2})$ such that $x - x'$ is an integer.
Now consider cones of the form
\[H(x) := \{(x',y') : |x' - x| < \frac{1}{2}, y > \frac{2}{\pi} \tan |\pi |x' - x||\}.\]
It was noted in \cite{astala2011random} that the field $x \mapsto \sqrt{2}\Wper(H(x))$ has formally the right covariance \eqref{eq:covariance_gff}, whence a natural sequence of approximation fields $(X_{1,n})$ is obtained by cutting the white noise at the level $1/n$. More precisely we define the truncated cones
\begin{equation}\label{eq:Ht}
H_{t}(x) := H(x) \cap \{(x,y) \in \reals^2 : y > e^{-t}\}
\end{equation}
and define the regular field $X_{1,n}$ by the formula
\begin{equation}\label{eq:X1}
X_{1,n}(x) := \sqrt{2} \Wper(H_{\log n}(x)).
\end{equation}

The third approximation fields $X_{3,n}$ are defined by using a H\"older continuous function $\varphi \in L^1(\reals)$ that satisfies $\int \varphi = 1$ and possesses the decay
\[|\varphi(x)| \le \frac{C}{(1 + |x|)^{1 + \delta}}\]
for some $C,\delta > 0$. We then set $X_{3,n} := \varphi_{1/n} * X_{per}$, where $X_{per}(x) = X(e^{2\pi i x})$ is the periodization of $X$ on $\reals$.
This form of convolution is fairly general, and encompasses convolutions against functions $\widetilde{\varphi}$ defined on the circle whose support do not contain the point $(-1,0)$.
\begin{example}
  Let $u$ be the harmonic extension of $X$ in the unit disc and consider the approximating fields $X_n(x) = u(r_n x)$ for $x \in S^1$ and for an increasing sequence of radii $r_n$ tending to $1$.
  Then $X_n(x)$ is obtained from $X$ by taking a convolution against the Poisson kernel $\varphi_{\varepsilon_n}$ on the real axis, where $\varphi(x) = \frac{2}{\pi(1 + 4\pi^2x^2)}$ and $\varepsilon_n = \log \frac{1}{r_n}$. This kind of approximations might be useful for example in studying fields that have been considered in \cite{miller2013quantum}.
\end{example}

The fourth approximation fields $X_{4,n}$ are defined by using a wavelet $\psi \colon \reals \to \reals$.
We assume that $\psi$ is obtained from a multiresolutional analysis, see \cite[Definition 2.2]{wojtaszczyk1997mathematical}, and that it has the decay
\begin{equation}\label{eq:wavelet_decay}
|\psi(x)| \le C (1 + |x|)^{-\alpha}
\end{equation}
with some constants $C > 0$ and $\alpha > 2$. We further assume that $\psi$ is of bounded variation, so that the distributional derivative $\psi'$ is a finite measure that satisfies the following condition on the tail
\begin{equation}\label{eq:wavelet_derivative_decay}
  \int_{-\infty}^\infty (1 + |x|) d|\psi'|(x) < \infty.
\end{equation}
\begin{remark}
  The conditions \eqref{eq:wavelet_decay} and \eqref{eq:wavelet_derivative_decay} are fairly general, especially the standard Haar wavelets satisfy them.
\end{remark}
With the above definitions it follows from \cite[Proposition 2.21]{wojtaszczyk1997mathematical} that the periodized wavelets
\[\psi_{j,k}(x) := 2^{j/2} \sum_{l=-\infty}^\infty \psi(2^j(x-l)-k)\]
together with the constant function $1$ form a basis for the space $L^2([0,1])$.

We next consider vaguelets that can be thought of as half-integrals of wavelets. Our presentation will be rather succinct -- another more detailed account can be found in the article by Tecu \cite{tecu2012random}. The vaguelet $\nu \colon \reals \to \reals$ is constructed by setting
\[\nu(x) := \frac{1}{\sqrt{2\pi}} \int_{-\infty}^\infty \frac{\psi(t)}{\sqrt{|x-t|}} \, dt.\]
An easy computation utilizing the decay of $\psi$ verifies that $\nu \colon \reals \to \reals$ satisfies 
\begin{equation}\label{eq:vaguelet_decay}
  |\nu(x)| \le \frac{C}{(1 + |x|)^{1 + \delta}}
\end{equation}
for some $C, \delta > 0$. We may then define the periodized functions
\begin{equation}\label{eq:periodized_vaguelet}
  \nu_{j,k}(x) := \sum_{l \in \integers} \nu(2^j(x - l) - k)
\end{equation}
for all $j \ge 0$ and $0 \le k \le 2^j - 1$.
It is straightforward to check that the Fourier coefficients of $\nu_{j,k}$ satisfy
\[\widehat{\nu}_{j,k}(n) = \frac{\widehat{\psi}_{j,k}(n)}{\sqrt{|2\pi n|}} \quad \text{when } n \neq 0,\]
and that $\nu_{j,k}$ is the half-integral of $\psi_{j,k}$ in the above sense.

The field $X_{4,n}$ can now be defined by
\begin{equation}\label{eq:X4n_definition}
  X_{4,n}(x) := 2\sqrt{\log 2} G + \sqrt{2\pi} \sum_{j=0}^n \sum_{k=0}^{2^j - 1} A_{j,k} \nu_{j,k}(x),
\end{equation}
where $G$ and $A_{j,k}$ are independent standard Gaussian random variables.
To see that this indeed has the right covariance one may first notice that
\[Y = \sum_{j=0}^\infty \sum_{k=0}^{2^j - 1} A_{j,k} \psi_{j,k}(x)\]
defines a distribution valued field satisfying $\E \langle Y, u \rangle \langle Y, v \rangle = \langle u, v\rangle$ for all 1-periodic $C^\infty$ functions $u$ and $v$. The field $X_{4,n}(x)$ is essentially the half integral of this field, whose covariance is given by
\[\E \langle I^{1/2} Y, u \rangle \langle I^{1/2} Y, v \rangle = \E \langle Y, I^{1/2} u \rangle \langle Y, I^{1/2} v\rangle = \langle I^{1/2} u, I^{1/2} v \rangle = \langle I u, v \rangle,\]
where the lift semigroup $I^{\beta} f$ for functions $f$ on $S^1$ is defined by describing its action on the Fourier basis: $I^\beta e^{2\pi i n x}=(2\pi|n|)^{-\beta} e^{2\pi i n x}$ for any $n \neq 0$ and $I^\beta 1 = 0$.
A short calculation shows that the operator $I$ has the right integral kernel $\frac{1}{\pi} \log \frac{1}{2|\sin(\pi (s-t))|}$.

\begin{proof}[Proof of Theorem~\ref{co:1}]
  The road map for the proof (as well as for the rest of the section) is as follows:
\begin{enumerate}
  \item We first show in Lemma~\ref{lemma:white_noise_limit} below that the chaos measures constructed from the white noise approximations converge weakly in $L^p$ by comparing it to the exactly scale invariant field on the unit interval by using Proposition~\ref{prop:adding_regular_field}.
  \item Next we verify in Lemma~\ref{lemma:fourier_limit} that the Fourier series approximations give the same result as the white noise approximations. This is done by a direct comparison of their covariances to verify the assumptions of Theorem~\ref{thm:uniqueness}.
  \item Thirdly we deduce in Lemma~\ref{lemma:convolution_limit} that convolution approximations also yield the same result by comparing a convolution against a Gaussian kernel to the Fourier series and again using Theorem~\ref{thm:uniqueness}.
  \item Fourthly we prove in Lemma~\ref{lemma:vaguelet_limit} that a vaguelet approximation yields the same result by comparing it against the white noise approximation.
  \item Finally, in Lemma~\ref{lemma:convg_in_prob} convergence in probability is established for the Fourier series, convolution and vaguelet approximations by invoking Theorem~\ref{thm:convergence_in_probability}.
\end{enumerate}
After the steps (1)--(5) the proof of Theorem~\ref{co:1} is complete.

The following lemma gives a quantitative estimate that can be used to compare fields defined using the hyperbolic white noise on $\H$.
\begin{lemma}\label{lemma:white_noise_holder}
  Let $U$ be an open subset of $\{(x,y) \in \H : y < 1\}$ such that the set $\{(x,y) \in U : y = s\}$ is an interval for all $0 < s < 1$.
  Let $f(s)$ denote the length of this interval and assume that $f(s) \le C s^{1 + \delta}$ for some $\delta > 0$.
  Then the map $(x,s) \mapsto \W(U_s+x)$ admits a modification that is almost surely continuous in $[a,b] \times [0,1]$ for any $a<b$, and almost surely the maps $x \mapsto \W(U_s + x)$ tend to $\W(U + x)$ uniformly when $s \to 0$.
  Here $U_s = \{(x,y) \in U : y > s\}$.
\end{lemma}

\begin{proof}
  Let us first show that
  \[\E|\W(U_s+x) - \W(U_s+y)|^2 \le \widetilde{C} |x-y|^{\frac{\delta}{1+\delta}}.\]
  for some $\widetilde{C} > 0$.
  By translation invariance of the covariance it is enough to consider $\E|\W(U_s + x) - \W(U_s)|^2$ and we can clearly assume that $0 < x < 1$.
  Obviously the $1$-measure of the set $((U_s + x) \cap \{y = a\}) \Delta (U_s \cap \{y = a\})$ equals $2\min(f(a), x)$.
  Hence we have
  \begin{align*}
    \E|\W(U_s + x) - \W(U_s)|^2 & = 2 \int_s^1 \frac{\min(f(y), x)}{y^2} \, dy \le 2 \max(1,C) \int_0^1 \frac{\min(y^{1 + \delta}, x)}{y^2} \, dy \\
                              & = 2 \max(1,C) \Big((1 + \delta^{-1}) x^{\frac{\delta}{1+\delta}} - x\Big) \le \widetilde{C} x^{\frac{\delta}{1 + \delta}}.
  \end{align*}
  Notice next that
  \[\E|\W(U_s) - \W(U_t)|^2 = \int_s^t \frac{f(u)}{u^2} \, du \le \frac{C}{\delta} (t^\delta - s^\delta).\]
  It follows that the map $(x,s) \mapsto \W(U_s + x)$ is H\"older-regular both in $x$ and $s$, and therefore also jointly.
  By Lemma~\ref{lemma:holder_regular_cont} the realizations can be chosen to be almost surely continuous in the rectangle $[a,b] \times [0,1]$ which obviously yields the claim.
\end{proof}

The claim concerning the approximating fields $X_{1,n}$ follows from the next lemma by taking into account the definitions \eqref{eq:Ht} and \eqref{eq:X1}. In the proof we identify the field on the unit circle locally as a perturbation of the exactly scaling field on the unit interval. For the chaos corresponding to the last mentioned field the corresponding convergence statement was proven in \cite{duplantier2014renormalization}, and we use this fundamental fact as the basis of the proof of the following lemma.
\begin{lemma}\label{lemma:white_noise_limit}
  Let either $\beta < 1$ and $\rho_t$ be the Lebesgue measure on the circle, or let $\beta = 1$ and $d\rho_t(x) = \sqrt{t} \, dx$.
  Then the measures
  \[e^{\beta \sqrt{2}\Wper(H_t(x)) - \beta^2 \E[\Wper(H_t(x))^2]} \, d\rho_t(x)\]
  defined on the unit circle {\rm (}which we identify with $\reals/\integers${\rm )} converge weakly in $L^p(\Omega)$ to a non-trivial measure $\mutorus{\beta}$ for $0 < p < 1$.
\end{lemma}

\begin{proof}
  As our starting point we know that the measures defined by
  \[d\widetilde{\mu}_t(x) := e^{\beta \sqrt{2} \W(A_t(x)) - \beta^2 \E[\W(A_t(x))^2]} \, d\rho_t(x)\]
  on the interval $[-\frac{1}{2},\frac{1}{2}]$ converge weakly in $L^p(\Omega)$ to a non-trivial measure for $0 < p < 1$ under the assumptions we have on $\beta$ and $\rho_t$. Here $A_t$ stands for the cone defined in \eqref{eq:At} in the introduction. One should keep in mind that we are using the same hyperbolic white noise when defining both $\W$ and $\Wper$.

  Let us split the cones $H_t$ into two sets $H_t^+$ and $H_t^-$, where
  \[H_t^+(x) := H_t(x) \cap \{(x,y) \in \H : y \ge 1\} \text{ and } H_t^-(x) := H_t(x) \cap \{(x,y) \in \H : y < 1\}.\]
  Clearly $\Wper(H_t(x)) = \Wper(H_t^+(x)) + \Wper(H_t^-(x))$ and by elementary geometry it is easy to see that if we restrict $x$ to the interval $(-\delta_0,\delta_0)$ where $\delta_0 = \frac{1}{2} - \frac{\arctan(\pi/2)}{\pi} \approx 0.18$, we have $(\Wper(H_t^-(x)))_{x \in (-\delta_0,\delta_0)} = (\W(H_t^-(x)))_{x \in (-\delta_0,\delta_0)}$. Hence our aim is to first verify the convergence on the interval $(-\delta_0, \delta_0)$.

  Write then $Y_t(x) = \Wper(H_t(x))$, $Y_t^+(x) = \Wper(H_t^+(x))$ and $Y_t^-(x) = \W(H_t^-(x))$ and similarly for the limit fields (which clearly exist in the sense of distributions) write $Y(x) = \Wper(H(x))$, $Y^+(x) = \Wper(H^+(x))$ and $Y^-(x) = \W(H^-(x))$.
  Let $X_t(x) := \W(A_t(x))$ and $X(x) := \W(A(x))$ and define $Z_t(x) := Y_t^-(x) - X_t(x)$ so that we may write $Y_t^-(x) = X_t(x) + Z_t(x)$.
  We next make sure that $Z_t(x)$ is a H\"older regular field, the realizations of which converge almost surely uniformly to the H\"older regular Gaussian field $Z(x) := Y^-(x) - X(x)$.

  The field $Z(x)$ decomposes into a sum $L(x) + R(x) + T(x)$, where $L(x) = -\W(\widetilde{L} + x)$, $R(x) = -\W(\widetilde{R} + x)$ and $T(x) = -\W(\widetilde{T} + x)$ with
  \begin{align*}
    \widetilde{R} & = \{(x,y) : \frac{1}{\pi} \arctan(\frac{\pi}{2} y) < x \le \frac{y}{2}, y < 1\} \\
    \widetilde{L} & = \{(-x,y) : (x,y) \in \widetilde{R}\} \\
    \widetilde{T} & = \{(x,y) : -\frac{1}{2} \le x \le \frac{1}{2}, y \ge 1\}.
  \end{align*}
  We define the truncated versions of $L_t$, $R_t$ and $T_t$ by cutting the respective sets at the level $e^{-t}$ as usual, so that $Z_t(x) = L_t(x) + R_t(x) + T_t(x)$. Clearly $T_t(x) = T(x)$ for $t \ge 0$.

  Let now $f(u) = \frac{u}{2} - \frac{1}{\pi} \arctan(\frac{\pi}{2} u)$.
  Using the Taylor series of $\arctan(u) = u - \frac{u^3}{3} + \frac{u^5}{5} - \frac{u^7}{7} + \dots$ we have
  \[f(u) = \frac{\pi^2}{24} u^3 + O(u^5),\]
  so $f(u) \le C u^3$ for some constant $C > 0$.
  It follows from Lemma~\ref{lemma:white_noise_holder} that $L_t(x)$ and $R_t(x)$ converge almost surely uniformly to the fields $L(x)$ and $R(x)$, so $Z_t(x)$ converges almost surely uniformly to $Z(x)$ as $t \to \infty$.

  Note that $\E[Z_t(x) X_t(x)]$ tends to a finite constant as $t \to \infty$, so the assumptions of Proposition~\ref{prop:adding_regular_field} are satisfied.
  Therefore the measures
  \[\nu_t = \int f(x) e^{\beta \sqrt{2} Y_t^-(x) - \beta^2 \E[Y_t^-(x)^2]} \, d\rho_t(x)\]
  on $(-\delta,\delta)$ converge weakly in $L^p(\Omega)$.
  Because $Y^+$ is a regular field, we may again use Proposition~\ref{prop:adding_regular_field} to conclude that also the measures
  \[\widetilde{\mu}_t(f) = \int f(x) e^{\beta \sqrt{2} Y_t(x) - \beta^2 \E[Y_t(x)^2]} \, d\rho_t(x)\]
  on $(-\delta,\delta)$ converge in $L^p(\Omega)$.
  By the translation invariance of the field the same holds for any interval of length $2\delta$.
  Let $I_1,\dots,I_n$ be intervals of length $2\delta$ that cover the unit circle and let $p_1,\dots,p_n \in C(S^1)$ be a partition of unity with respect to the cover $I_k$.
  The measure
  \[\mu_t(f) = \int f(x) e^{\beta \sqrt{2} Y_t(x) - \beta^2 \E[Y_t(x)^2]} \, d\rho_t(x)\]
  on the whole unit circle can be expressed as a sum $d\mu_t(x) = p_1(x) d\widetilde{\mu}^{(1)}_t(x) + \dots + p_2(x) d\widetilde{\mu}^{(n)}_t(x)$.
  Because each of the summands converges in $L^p(\Omega)$, we see that the also the family of measures $\mu_t$ converges in $L^p(\Omega)$.
\end{proof}

\begin{lemma}\label{lemma:fourier_limit}
  Let either $\beta < 1$ and $d\rho_n(x) = dx$ for all $n \ge 1$ or let $\beta = 1$ and $d\rho_n(x) = \sqrt{\log n} \, dx$.
  Then the measures
  \[d\mu_{2,n}(x) := e^{\beta X_{2,n}(x) - \frac{\beta^2}{2} \E[X_{2,n}(x)^2]} \, d\rho_n(x)\]
  converge in distribution to the random measure $\mutorus{\beta}$ constructed in Lemma~\ref{lemma:white_noise_limit}.
\end{lemma}

\begin{proof}
  Let $f_n(x) := \E[X_{2,n}(x)X_{2,n}(0)]$. It is straightforward to calculate that
  \[f_n(x) = 4 \log 2 + 2 \sum_{k=1}^n \frac{\cos(2\pi k x)}{k}.\]
  In particular $f_n(0) = 4 \log 2 + 2 H_n$, where $H_n$ is the $n$th Harmonic number, $H_n = \log n + \gamma + O(\frac{1}{n})$ with $\gamma$ being the Euler--Mascheroni constant.
  Let $f(x) := 4 \log 2 + 2 \log \frac{1}{2|\sin(\pi x)|}$ be the limit covariance and define $g_n(x) := f(x) - f_n(x)$.
  One can easily compute that for $0 < x \le \frac{1}{2}$ we have
  \[g_n'(x) = -\frac{2 \pi \cos(2\pi(n + \frac{1}{2})x)}{\sin(\pi x)}.\]
  In particular the maximums and minimums of the difference $g_n(x)$ occur at the points $x^{(n)}_j = \frac{2j+1}{4n+2}$, $0 \le j \le n$.
  Consider the telescoping sum
  \begin{equation}\label{eq:telescoping}
    g_n(x^{(n)}_j) = (g_n(x^{(n)}_j) - g_n(x^{(n)}_{j+1})) + \dots + (g_n(x^{(n)}_{n-1}) - g_n(x^{(n)}_n)) + g_n(x^{(n)}_n).
  \end{equation}
  Here the terms in parentheses form an alternating series whose terms are decreasing in absolute value. Moreover, the term $g_n(x^{(n)}_0) - g_n(x^{(n)}_1)$ stays bounded as $n \to \infty$ and the term $g_n(x^{(n)}_n)$ goes to $0$ as $n \to \infty$.
  All this is obvious from writing
  \begin{align}\label{eq:g_n_integral}
    g_n(x^{(n)}_{j+1}) - g_n(x^{(n)}_j) & = \int_{x^{(n)}_j}^{x^{(n)}_{j+1}} g_n'(t) \, dt = -2\pi \int_{\frac{2j+1}{4n+2}}^{\frac{2j+3}{4n+2}} \frac{\cos(\pi(2n+1) t)}{\sin(\pi t)} \, dt \\
                                        & = \frac{-2\pi}{2n+1} \int_{-1/2}^{1/2} \frac{\cos(\pi(y+j+1))}{\sin(\pi \frac{y+j+1}{2n+1})} \, dy \nonumber \\
                                        & = \frac{(-1)^j 2 \pi}{2n + 1} \int_{-1/2}^{1/2} \frac{\cos(\pi y)}{\sin(\pi \frac{y+j+1}{2n+1})} \, dy, \nonumber \\
    g_n(x^{(n)}_n) & = -2\log(2) - 2\sum_{k=1}^n \frac{(-1)^k}{k}. \nonumber
  \end{align}
  In particular we deduce that
  \begin{equation}\label{eq:g_n_bounded}
  \sup_{n \ge 1} \sup_{x \ge x^{(n)}_0} |g_n(x)| < \infty.
  \end{equation}
  Notice also that for any fixed $\varepsilon > 0$ all the maximums and minimums in the range $x > \varepsilon$ are located at the points $x^{(n)}_j$ with $j > 2\varepsilon n + \varepsilon - \frac{1}{2}$, and
  \[\lim_{n \to \infty} \sup_{j > \varepsilon n + \varepsilon - \frac{1}{2}} |g_n(x^{(n)}_{j+1}) - g_n(x^{(n)}_j)| = 0\]
  by \eqref{eq:g_n_integral}. From \eqref{eq:telescoping} it follows that the Fourier covariance converges to the limit covariance uniformly in the set $\{|x| > \varepsilon\}$, a fact that could also be deduced from the localized uniform convergence of the Fourier series of smooth functions.

  Consider next the white noise covariance $h_t(x) := 2 \E[\Wper(H_t(x)) \Wper(H_t(0))]$. By symmetry we may assume all the time that $x>0$. After a slightly tedious calculation one arrives at the formula
  \[h_t(x) = \begin{cases} 4 \log 2 + 2 \log \frac{1}{2\sin(\pi x)}, & \text{if } x > \frac{2}{\pi} \arctan(\frac{\pi}{2} e^{-t}) \\
                           -2 x e^t + 2 t - 2 \log(\cos(\frac{\pi}{2} x)) + \log(\pi^2 e^{-2t} + 4) \\
                           + \frac{2 \arctan(\frac{\pi}{2} e^{-t})}{\frac{\pi}{2} e^{-t}} - 2 \log(\pi), & \text{if } x \le \frac{2}{\pi} \arctan\big(\frac{\pi}{2} e^{-t}\big).
             \end{cases}\]
   Let us consider the approximation along the sequence $t_n = \log(n)$.
   Then $h_t(0) = 2\log(n) + O(1)$.
   Moreover at the point $x_n = \frac{2}{\pi} \arctan(\frac{\pi}{2} e^{-t_n}) = \frac{2}{\pi} \arctan(\frac{\pi}{2n})$ we have
   \[h_{t_n}(x_n) = 4 \log 2 + 2 \log \frac{1}{2\sin(2 \arctan(\frac{\pi}{2n}))} = 2 \log(n) + O(1).\]
   Because the function $h_{t_n}$ without the bounded term $-2 \log(\cos(\frac{\pi}{2} x))$ is linear and decreasing on the interval $[0,x_n]$ we know that it is actually $2 \log(n) + O(1)$ on that whole interval.
   Similarly it is easy to check that for the Fourier series we have $f_n(x) = 2 \log(n) + O(1)$ on the interval $[0,x_n]$
   because $|f'_n(x)| \le 4\pi n$ and $x_n = O(\frac{1}{n})$. Thus $|f_n(x) - h_{t_n}(x)| = O(1)$ for $x \le x_n$. For $x \ge x_n$ it follows from \eqref{eq:g_n_bounded} that $|f_n(x) - h_{t_n}(x)| = |g_n(x)|$ is bounded.
   
   From the above considerations and symmetry it follows that the covariances of the fields $X_{1,n}$ and $X_{2,n}$ satisfy the assumptions of Theorem~\ref{thm:uniqueness}. This finishes the proof.
\end{proof}

\begin{remark}
  The somewhat delicate considerations in the previous proof are necessary because of the fairly unwieldy behaviour of the Dirichlet kernel.
\end{remark}

Next we verify that any convolution approximation to the field $X$ also has the same limit.

\begin{lemma}\label{lemma:convolution_limit}
  Let $\varphi$ be a H\"older continuous mollifier satisfying $\int_{-\infty}^\infty \varphi(x) \, dx = 1$ and $\varphi(x) = O(x^{-1-\delta})$ for some $\delta > 0$.
  Then the fields $X_{3,n}$ defined on $S^1$ by using the periodized field on $\reals${\rm :}
  \[X_{3,n}(x) := (\varphi_{1/n} * X_{per})(x)\]
  are H\"older-regular and the measures
  \[d\mu_{3,n} := e^{\beta X_{3,n}(x) - \frac{\beta^2}{2} \E[X_{3,n}(x)^2]} \, d\rho_n(x),\]
  converge in distribution to $\mutorus{\beta}$.
  Here $\rho_n$ is the Lebesgue measure if $\beta < 1$ and $d\rho_n = \sqrt{\log n} \, dx$ if $\beta = 1$.
\end{lemma}

\begin{proof}
  It is enough to show the assumptions of Theorem~\ref{thm:uniqueness} for one kernel satisfying the conditions of the lemma because of Corollary~\ref{cor:convolutions_circle}, and because of Lemma~\ref{lemma:fourier_limit} we can do our comparison against the covariance obtained from the Fourier series construction.
  We will make the convenient choice of $\varphi(x) = \frac{1}{\sqrt{2\pi}} e^{-\frac{x^2}{2}}$ as our kernel.
  The covariance of the field $\varphi_\varepsilon * X_{per}$ is given by $(\psi_\varepsilon * f)(x-y)$,
  where $\psi_\varepsilon(x) = (\varphi_\varepsilon * \varphi_\varepsilon(- \cdot))(x) = \frac{1}{2\varepsilon\sqrt{\pi}} e^{-\frac{x^2}{4\varepsilon^2}}$ and $f(x) = 4 \log 2 + 2 \log \frac{1}{2|\sin(\pi x)|}$.

  Using the identity $\log \frac{1}{2|\sin(\pi x)|} = \sum_{k=1}^\infty \frac{\cos(2\pi k x)}{k}$, a short computation shows that we can write the difference of the covariances of $X_{2,n}$ (the Fourier field) and $X_{3,n}$ in the form
  \[2\sum_{k=1}^n \frac{\cos(2\pi k x)}{k}(1 - e^{-4\pi^2 \frac{k^2}{n^2}}) - 2\sum_{k=n+1}^\infty \frac{\cos(2\pi k x)}{k} e^{-4\pi^2 \frac{k^2}{n^2}}.\]
  Since $1 - e^{-x} \le x$ for $x \ge 0$, the first term is bounded by $2 \sum_{k=1}^n \frac{4\pi^2 k}{n^2} \le 16\pi^2$.
  In turn the second term is bounded from above by
  \[2\int_{n}^\infty \frac{e^{-4\pi^2 \frac{t^2}{n^2}}}{t} \, dt = 2\int_{1}^\infty \frac{e^{-4\pi^2 s^2}}{s} \, ds.\]
  Because both of the covariances converge locally uniformly outside the diagonal, we again see that the assumptions of Theorem~\ref{thm:uniqueness} are satisfied.
\end{proof}

Our next goal is to prove the convergence in distribution for the vaguelet approximation $X_{4,n}$. We start with the following useful estimate.

\begin{lemma}\label{lemma:half_integral_estimate}
  Let $f \colon \reals \to \reals$ be a bounded integrable function and let
  \[F(x) = \frac{1}{\sqrt{2\pi}} \int_{-\infty}^\infty \frac{f(t)}{\sqrt{|x-t|}} \, dt\]
  be its half-integral. Then there exists a constant $C > 0$ {\rm (}not depending on $f${\rm )} such that for all $x,y \in \reals$ we have
  \[|F(x) - F(y)| \le C \|f\|_\infty \sqrt{|x-y|}.\]
\end{lemma}

\begin{proof}
  Clearly it is enough to show that
  \[\int_{-\infty}^\infty \left| \frac{1}{\sqrt{|x-t|}} - \frac{1}{\sqrt{|y-t|}} \right| \, dt \le C \sqrt{|x-y|}.\]
  Notice that the integrand can be approximated by
  \begin{align*}
    \left| \frac{1}{\sqrt{|x-t|}} - \frac{1}{\sqrt{|y-t|}} \right| & = \frac{\big||y-t| - |x-t|\big|}{|x-t|\sqrt{|y-t|} + \sqrt{|x-t|}|y-t|} \\
                                                            & \le \frac{|x-y|}{|x-t|\sqrt{|y-t|} + \sqrt{|x-t|}{|y-t|}}.
  \end{align*}
  We can without loss of generality assume that $x < y$ and split the domain of integration to the intervals $(-\infty, x]$, $[x,\frac{x+y}{2}]$, $[\frac{x+y}{2}, y]$ and $[y,\infty)$. On each of the intervals the value of the integral is easily estimated to be less than some constant times $\sqrt{|x-y|}$, which gives the result.
\end{proof}

In the lemmas below we recall the definition of the field $X_{4,n}$ in \eqref{eq:X4n_definition}.

\begin{lemma}\label{lemma:vaguelet_periodized_decay}
  We have
  \[\nu_{j,0}(x) \le \frac{c}{(1 + 2^{j} \dist(x, 0))^{1 + \delta}}\]
  for some constant $c > 0$. Here $\dist(x,y) = \min \{|x - y + k| : k \in \integers\}$.
\end{lemma}

\begin{proof}
  Without loss of generality we may assume that $0 \le x < 1$ and let $d = \dist(x,0)$.
  We have
  \begin{align*}
    |\nu_{j,0}(x)| \le & \sum_{l \in \integers} \frac{C}{(1 + 2^j |x-l|)^{1+\delta}} \le \sum_{l=0}^\infty \frac{C}{(1 + 2^j x + 2^j l)^{1 + \delta}} + \sum_{l=1}^\infty \frac{C}{(1 + 2^j l - 2^j x)^{1 + \delta}} \\
    \le & \frac{2C}{(1 + 2^j d)^{1 + \delta}} + 2\sum_{l=1}^\infty \frac{C}{(1 + 2^j l)^{1 + \delta}} \\
  \le & \frac{2C}{(1 + 2^j d)^{1 + \delta}} + \frac{2C}{(1 + 2^j)^{1 + \delta}} + 2 \int_1^\infty \frac{C}{(1 + 2^j u)^{1 + \delta}} \, du \\
    \le & 4C \frac{1}{(1 + 2^j d)^{1 + \delta}} + \frac{2C}{\delta} \frac{1}{2^j (1 + 2^j)^{\delta}} \le \frac{c}{(1 + 2^jd)^{1 + \delta}}.\qedhere
  \end{align*}
\end{proof}

\begin{lemma}\label{lemma:vaguelet_abs_sum}
  There exists a constant $A > 0$ such that
  \[\sum_{k=0}^{2^j - 1} |\nu_{j,k}(x)| \le A\]
  for all $j \ge 0$ and $x \in \reals$.
\end{lemma}

\begin{proof}
  By using Lemma~\ref{lemma:vaguelet_periodized_decay} and the fact that $\nu_{j,k}(x) = \nu_{j,0}(x - k 2^{-j})$ we have
  \begin{align*}
    \sum_{k=0}^{2^j - 1} |\nu_{j,k}(x)| = & \sum_{k=0}^{2^j - 1} |\nu_{j,0}(x - k 2^{-j})| \le \sum_{k=0}^{2^j - 1} \frac{c}{(1 + 2^j \dist(x - k 2^{-j},0))^{1 + \delta}} \\
    \le & 2c \sum_{k=0}^{\infty} \frac{1}{(1 + k)^{1 + \delta}} < \infty.\qedhere
  \end{align*}
\end{proof}

\begin{lemma}\label{lemma:vaguelet_product_sum}
  There exists a constant $B > 0$ such that for all $n \ge 0$ and $x,y \in \reals$ satisfying $\dist(x,y) \ge 2^{-n}$ we have
  \[\sum_{j=n}^{\infty} \sum_{k=0}^{2^j - 1} |\nu_{j,k}(x) \nu_{j,k}(y)| \le B.\]
\end{lemma}

\begin{proof}
  By using Lemma~\ref{lemma:vaguelet_periodized_decay} and the fact that $\nu_{j,k}(x) = \nu_{j,0}(x - k 2^{-j})$ we may estimate
  \begin{align*}
    & \sum_{j=n}^\infty \sum_{k=0}^{2^j - 1} |\nu_{j,k}(x) \nu_{j,k}(y)| \\
    \le & \sum_{j=n}^\infty \sum_{k=0}^{2^j - 1} \frac{c^2}{(1 + 2^j \dist(x - k 2^{-j},0))^{1+\delta} (1 + 2^j \dist(y - k 2^{-j},0))^{1+\delta}} \\
    \le & \sum_{j=n}^\infty \sum_{k=0}^{2^j - 1} \frac{2c^2}{\max((1 + k)^{1+\delta}, (1 + 2^{j-n-1})^{1 + \delta})} \\
    \le & 2c^2 \sum_{j=n}^\infty \Big( \frac{1 + 2^{j-n-1}}{(1 + 2^{j-n-1})^{1 + \delta}} + \sum_{k=2^{j-n-1} + 1}^\infty \frac{1}{(1 + k)^{1+\delta}} \Big) \\
    \le & 2c^2 \sum_{j=n}^\infty \Big( \frac{1}{(1 + 2^{j-n-1})^{\delta}} + \int_{2^{j-n-1}}^\infty \frac{1}{(1 + x)^{1+\delta}} \, dx \Big) \\
    \le & 2c^2(1+\frac{1}{\delta}) \sum_{j=n}^\infty \frac{1}{(1 + 2^{j-n-1})^\delta} \le 2c^2(1 + \frac{1}{\delta}) \sum_{j=0}^\infty 2^{-\delta(j-1)} = B < \infty.\qedhere
  \end{align*}
\end{proof}

\begin{lemma}\label{lemma:vaguelet_limit}
  Let either $\beta < 1$ and $d\rho_n(x) = dx$ for all $n \ge 1$ or let $\beta = 1$ and $d\rho_n(x) = \sqrt{n\log 2} \, dx$. Then the measures
  \[d\mu_{4,n} := e^{\beta X_{4,n}(x) - \frac{\beta^2}{2} \E[X_{4,n}(x)^2]} \, d\rho_n(x)\]
  converge in distribution to the random measure $\mutorus{\beta}$ constructed in Lemma~\ref{lemma:white_noise_limit}.
\end{lemma}

\begin{proof}
  The covariance $C_n(x,y)$ of the field $X_{4,n}$ is given by
  \[C_n(x,y) = 4 \log 2 + 2\pi \sum_{j=0}^n \sum_{k=0}^{2^j - 1} \nu_{j,k}(x) \nu_{j,k}(y).\]

  Let $\psi_{j,k}$ be the periodized wavelets. Then there exists a constant $D > 0$ such that $\|\psi_{j,k}\|_\infty \le D 2^{j/2}$ for all $j \ge 0$, $0 \le k \le 2^j - 1$.
  It follows from Lemma~\ref{lemma:half_integral_estimate} and Lemma~\ref{lemma:vaguelet_abs_sum} that when $\dist(x,y) \le 2^{-n}$, we have
  \begin{align}\label{eq:vaguelet_covariance_diagonal}
    |C_n(x,x) - C_n(x,y)| \le & 2\pi \sum_{j=0}^n \sum_{k=0}^{2^j - 1} |\nu_{j,k}(x)| |\nu_{j,k}(x) - \nu_{j,k}(y)| \\
    \le & 2\pi C \sqrt{|x-y|} \sum_{j=0}^n \sum_{k=0}^{2^j - 1} |\nu_{j,k}(x)| \|\psi_{j,k}\|_\infty \nonumber \\
    \le & 2\pi A C D \sqrt{|x-y|} \sum_{j=0}^n 2^{j/2} \le E. \nonumber
  \end{align}
  for some constant $E > 0$. It follows from Lemma~\ref{lemma:vaguelet_product_sum} that for any $\varepsilon > 0$ the covariances $C_n(x,y)$ converge uniformly in the set $V_\varepsilon = \{(x,y) : \dist(x,y) \ge \varepsilon\}$. Obviously by definition there is a distributional convergence to the right covariance $4 \log 2 + 2 \log \frac{1}{2 |\sin(\pi (x-y))|}$ and this must agree with the uniform limit in $V_\varepsilon$. Especially, by invoking again the bound from Lemma~\ref{lemma:vaguelet_product_sum} we deduce that
  \begin{equation}\label{eq:vaguelet_covariance_at_point}
    |C_n(x,x + 2^{-n}) - 4 \log 2 - 2 \log \frac{1}{2\sin(\pi 2^{-n})}| \le 2\pi B.
  \end{equation}
  Thus by combining \eqref{eq:vaguelet_covariance_diagonal} and \eqref{eq:vaguelet_covariance_at_point} the covariance satisfies
  \[|C_n(x,y) - 2 n \log 2| \le F \quad \text{for all } (x,y) \in \{(x,y) : \dist(x,y) \le 2^{-n}\}\]
  for some constant $F > 0$.
  From the known behaviour (see e.g.\ the end of the proof of Lemma~\ref{lemma:fourier_limit}) of the covariance of the white noise field $X_{1,n}$ it is now easy to see that the assumptions of Theorem~\ref{thm:uniqueness} are satisfied for the pair $(X_{4,n})$ and $(X_{1,n})$.
\end{proof}

Finally we observe that the convergence in lemmas~\ref{lemma:fourier_limit},~\ref{lemma:convolution_limit} and~\ref{lemma:vaguelet_limit} also takes place weakly in $L^p$.
\begin{lemma}\label{lemma:convg_in_prob}
  The convergences stated in lemmas~\ref{lemma:fourier_limit},~\ref{lemma:convolution_limit} and~\ref{lemma:vaguelet_limit} take place in $L^p$ for $0 < p < 1$ {\rm (}especially in probability{\rm )}.
\end{lemma}
\begin{proof}
  We only prove the claim in the critical case since the subcritical case is similar.
  We will use the fields $X_{1,n}$ as the fields $X_n$ in Theorem~\ref{thm:convergence_in_probability}.
Then according to Lemma~\ref{lemma:white_noise_limit} we have that $e^{X_n - \frac{1}{2}\E[X_n^2]} \, d\rho_n$ converges in probability to a measure $\mutorus{1}$ when $d\rho_n = \sqrt{\log n} \, dx$.

  In the case of the Fourier approximation we can define $R_n$ in Theorem~\ref{thm:convergence_in_probability} to be the $n$th partial sum of the Fourier series. That is
  \[R_n f := \sum_{k=-n}^n \widehat{f}(k) e^{2\pi i k x}.\]
  Recalling Jackson's theorem on the uniform convergence of Fourier series of H\"older continuous functions, it is straightforward to check that $R_n$ is a linear regularization process.

  In the case of convolutions we take $R_n$ to be the convolution against $\frac{1}{\varepsilon_n} \varphi(\frac{x}{\varepsilon_n})$, where $(\varepsilon_n)_{n \ge 1}$ is a sequence of positive numbers tending to $0$. The sequence $(R_n)$ obviously satisfies the required conditions.

Finally, we sketch the proof for the vaguelet approximations. This time we employ the  sequence of operators
   \[
     R_n f (x) :=\int_0^1f+\sum_{j=0}^n\sum_{k=0}^{2^j-1}\left(\int_0^1\psi_{j,k}(y)\big(I^{-1/2}f(y)\big)dy\right)\nu_{j,k}(x).
   \]
   Because of finiteness of the defining series it is easy to see that $(R_n)$ satisfies the second  condition in Definition~\ref{def:linear_reg}. For the first condition we first fix $\alpha \in (0,1/2)$ and observe that $R_n \nu_{j',k'} = \nu_{j',k'}$ as soon as $n \ge j'$. By the density of vaguelets, in order to verify the first condition it is enough to check that the remainder term tends uniformly to $0$ for any $f \in C^\alpha(S^1)$. We begin by noting that $\frac{d}{dx}=- i H I^{-1}$, where $H$ is the Hilbert transform, which yields for $f\in C^\alpha(S^1)$ 
   \[
  \Big| \int_0^1\psi_{j,k}(y)\big(I^{-1/2}f(y)\big)\Big| = \Big| \int_0^1\frac{d}{dy}\psi_{j,k}(y)\big(HI^{+1/2}f(y)\big)\Big|\leq C2^{-\alpha j},\qquad x\in [0,1),
   \]
   since $HI^{+1/2}f(x)\in C^{\alpha +1/2}(S^1)$ by the standard mapping properties of $I^\beta$, and  the Hilbert transform is bounded on any of the $C^\alpha$-spaces. Above, the final estimate was obtained by computing for any $g \in C^{\alpha + 1/2}(S^1)$ with periodic continuation $G$ to $\reals$ that
   \begin{align*}
     \left|\int_0^1 \frac{d}{dx}\psi_{j,0}(x) g(x)\right| & = \left|\int_{-\infty}^\infty 2^{\frac{3}{2}j} d\psi'(2^j x) G(x)\right| = 2^{j/2} \left|\int_{-\infty}^\infty d\psi'(x) (G(2^{-j} x) - G(0))\right| \\
                                                          & \le 2^{j/2} \int_{-\infty}^\infty |d\psi'(x)| (2^{-j} x)^{\alpha + 1/2} \le 2^{-\alpha j} \int_{-\infty}^\infty |d\psi'(x)| (1 + |x|).
   \end{align*}
   The last integral is finite by the assumption \eqref{eq:wavelet_derivative_decay}. Together with Lemma~\ref{lemma:vaguelet_abs_sum} this obviously yields the desired uniform convergence.

   The proofs of the lemmas~\ref{lemma:fourier_limit},~\ref{lemma:convolution_limit} and~\ref{lemma:vaguelet_limit} show that the covariances stay at a bounded distance from the covariance of the field $X_{1,n}$, and therefore a standard application of Kahane's convexity inequality gives us an $L^p$ bound. Combining this with Theorem~\ref{thm:convergence_in_probability} yields the result.
\end{proof}

As noted in the beginning of this section, having proved all the lemmas above we may conclude the proof of Theorem~\ref{co:1}.
\end{proof}

\begin{remark}
  In the case of vaguelet approximations we may also rewrite
  \[X(x) = \sum_{i=1}^\infty \widetilde{A}_i \widetilde{\nu}_i(x),\]
  where $\widetilde{A}_i$ and $\widetilde{\nu}_i$ are the random coefficents and vaguelets appearing in \eqref{eq:X4n_definition} ordered in their natural order. The convergence and uniqueness then also holds for the chaos constructed from the fields
  \[\widetilde{X}_{4,n} := \sum_{i=1}^n \widetilde{A}_i \widetilde{\nu}_i(x),\]
  with the normalizing measure $d\rho_n(x) = \sqrt{\log n} \, dx$.
\end{remark}

\begin{remark}
  There are many interesting questions that we did not touch in this paper. For example (this question is due to Vincent Vargas), it is natural to ask whether the convergence or uniqueness of the derivative martingale\cite{duplantier2014critical} depends on the approximations used.
\end{remark}

\end{document}